\documentclass[a4paper, 12pt]{amsart}
\usepackage[utf8]{inputenc}
\usepackage[english]{babel}
\usepackage[T1]{fontenc}
\usepackage{amsmath, amssymb, amsfonts}
\usepackage[all]{xy}
\usepackage{enumerate}
\usepackage{graphicx}
\usepackage[hmargin=2cm, vmargin=2.5cm]{geometry}
\usepackage{rotate}
\usepackage{MnSymbol}
\usepackage{tikz}
\usepackage{comment}
\usepackage{hyperref}

\title[Noncrossing partition quantum groups]{On two-coloured noncrossing partition quantum groups}
\author{Amaury Freslon}
\address{Laboratoire de Math\'ematiques d'Orsay, Univ. Paris-Sud, CNRS, Universit\'e Paris-Saclay, 91405 Orsay, France}
\email{amaury.freslon@math.u-psud.fr}
\date{}

\theoremstyle{plain}
\newtheorem{thm}{Theorem}[section]
\newtheorem{prop}[thm]{Proposition}

\newtheorem{lem}[thm]{Lemma}

\theoremstyle{definition}
\newtheorem{de}[thm]{Definition}

\theoremstyle{remark}
\newtheorem{rem}[thm]{Remark}

\DeclareMathOperator{\Mor}{Mor}
\DeclareMathOperator{\Id}{Id}
\DeclareMathOperator{\id}{id}

\DeclareMathOperator{\Proj}{Proj_{\CC}}

\DeclareMathOperator{\rl}{rl}
\DeclareMathOperator{\Span}{Span}

\newcommand{\A}{\mathcal{A}}

\newcommand{\C}{\mathbb{C}}
\newcommand{\CC}{\mathcal{C}}
\newcommand{\D}{\Delta}
\newcommand{\DD}{\mathcal{D}}

\newcommand{\G}{\mathbb{G}}
\newcommand{\Gr}{\mathcal{G}}
\newcommand{\HH}{\mathbb{H}}

\newcommand{\N}{\mathbb{N}}

\newcommand{\Z}{\mathbb{Z}}

\newcommand{\co}{\overline}

\newcommand{\ii}{\imath}

\newcommand{\idpart}{|}

\begin{document}

\begin{abstract}
We classify compact quantum groups associated to noncrossing partitions coloured with two elements $x$ and $y$ which are their own inverses. Together with the work of P. Tarrago and M. Weber, this completes the classification of all noncrossing partition quantum groups on two colours. We also give some general results on the class of all noncrossing partition quantum groups and suggest some wider classification statements.
\end{abstract}

\maketitle

\section{Introduction}

T. Banica and R. Speicher developped in \cite{banica2009liberation} a combinatorial machinery to build families of compact quantum groups. The idea is to start with a suitable collection of partitions of finite sets, to build a tensor category out of it and then to use a Tannaka-Krein type argument to obtain a compact quantum group (called an \emph{easy quantum group}). This idea proved very fruitful and was followed by many works both exploring the construction and generalizing it. One of the interesting features of the construction is that it potentially contains new examples of compact quantum groups. Another one is that the objects thus obtained are naturally linked to (free) probability, in particular through De Finetti theorems (see for instance \cite{banica2012finetti}).

Our main interest in the present paper lies in quantum groups built from partitions which are \emph{noncrossing} (see Definition \ref{de:noncrossing}). It turns out that in this setting, the theory of T. Banica and R. Speicher gives a unified construction of several important examples of compact quantum groups studied up to now. While it was known from the beginning that the original construction would not give anything more for noncrossing partitions, it soon became clear that it could be extended in several ways (see for instance \cite{tarrago2015unitary}, \cite{cebron2016quantum} and \cite{speicher2016quantum}). In \cite{freslon2014partition}, we introduced a general setting called \emph{partition quantum groups} using coloured partitions to generalize easy quantum groups. We were in particular motivated by the following question of T. Banica and R. Vergnioux in \cite{banica2009fusion} : can all free quantum groups (in the sense of \cite[Def 10.2]{banica2009fusion}) be described by partitions ? The affirmative answer relies on the crucial fact that the lack of one-dimensional representations (which is closely linked to freeness) translates into a stability property for the collection of partitions, so that we were then able to classify them. In that way, we obtained as a by-product a model for all possible compact quantum groups with free fusion semiring. The present work complements \cite{freslon2014partition} since we will now deal with the non-free case.

Such results naturally rise the question of classifying all partition quantum groups associated to noncrossing partitions. Let us recall what is known on this problem. The case of orthogonal easy quantum groups was completed by S. Raum and M. Weber in \cite{weber2012classification} and \cite{raum2013full} (and the authors in fact classified \emph{all} easy quantum groups, not only the noncrossing ones). In the unitary case, things are more involved but P. Tarrago and M. Weber were able to classify all noncrossing unitary easy quantum groups in \cite{tarrago2015unitary} and \cite{tarrago2018classification}. In Sections \ref{sec:orthogonal} to \ref{sec:symmetric} we will classify all noncrossing partition quantum groups on two colours which are their own inverses. Together with the results above, this completes the classification of noncrossing partition quantum groups on one or two colours. Beyond these results, the aim of this paper is to highlight two points concerning the general classification problem.

The first point is the method. In all previous works, the classification is done by combinatorial arguments on the set of partitions so that the associated quantum groups never enter the picture. Here we take another point of view since we are more interested in classifying the quantum groups rather than the underlying categories of partitions. We therefore use a different strategy relying on the results of \cite{freslon2013representation} linking the partitions with the representation theory of the associated quantum group (see Subsection \ref{subsec:projective} for details). Considering a category of partitions $\CC$, we first build a subcategory of partitions $\CC'\subset \CC$ satisfying two properties :
\begin{itemize}
\item it is simpler in the sense that we can easily describe the associated compact quantum group,
\item it is large in the sense that any projective partition (see Definition \ref{de:projective}) of $\CC$ lies in $\CC'$.
\end{itemize}
Then, any partition $p\in \CC\setminus\CC'$ is equivalent, when rotated on one line, to the triviality of a one-dimensional representation of the quantum group associated to $\CC'$. Thus, the quantum group associated to $\CC$ is simply the quotient of the one associated to $\CC'$ by some relations which can be expressed in the group of one-dimensional representations (we will call these \emph{group-like relations} in Definition \ref{de:grouplike}).

The second point is the way of stating the classification. Quotienting by a group-like relation means that the C*-algebra $C(\G_{N}(\CC))$ is quotiented by the Hopf $*$-ideal generated by some group-like element. Such an ideal (and the corresponding quotient) may be very large and difficult to describe explicitly. This means that our classification will not be a list of all possible cases but rather a list of relations by which one may quotient. This can seem disappointing compared to the other results available so far, but it is possible to push our work further to produce the desired list. However, such a list would be long and perhaps not enlightening. More importantly, this approach would become intractable when the number of colours increases. On the contrary, our description is amenable to generalizations and we suggest in the end of this article how it could be extended to give a description of the class of all noncrossing partition quantum groups.

Let us end this introduction with an outline of the paper. In Section \ref{sec:preliminaries} we introduce some notations and recall the main results needed in the paper. Then, we describe in Section \ref{sec:structure} some basic operations preserving the class of noncrossing partition quantum groups. We also introduce a new family of compact quantum groups called free wreath products of pairs, which will be important in Section \ref{sec:hyperoctahedral}. The classification of noncrossing partition quantum groups on two colours is split into four parts : in Sections \ref{sec:orthogonal} and \ref{sec:bistochastic} we deal with categories of partitions where all blocks have size less than two and in Sections \ref{sec:hyperoctahedral} and \ref{sec:symmetric} we treat the other cases. We end with Section \ref{sec:summary} where we summarize our results and restate them in a way which makes sense for more colours.

\subsection*{Acknowledgment}

We thank the referee for comprehensively reading this work and making comments and suggestions which greatly improved the quality of this article. We are also grateful to Moritz Weber for his comments on a earlier version.

\section{Preliminaries}\label{sec:preliminaries}

In this section we introduce some basic material, mainly to fix notations. The reader may refer to \cite{freslon2013representation} for more background on the link between partitions and quantum groups and to the book \cite{neshveyev2014compact} for details and proofs concerning the theory of compact quantum groups.

\subsection{Coloured partitions}

The main topic of this work is partitions of finite sets. Even though these may seem to be very simple objects, they exhibit rich combinatorial properties. These properties are best seen using a graphical representation of the partitions. Let us denote by $P(k, l)$ the set of partitions of the set $\{1, 2, \dots, k+l\}$. We represent such partitions in the following way : we draw a row of $k$ points above a row of $l$ points and then connect the points which belong to the same subset of the partition. Here is for instance the representation of $p = \{\{1, 5, 6, 7\}, \{2, 4\}, \{3\}\}\in P(5, 2)$ :
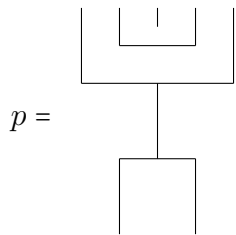
\begin{center}
\begin{tikzpicture}[scale=0.5]
\draw (0,-1) -- (0,1);
\draw (-2,1) -- (2,1);
\draw (-2,1) -- (-2,3);
\draw (2,1) -- (2,3);
\draw (-1,2) -- (-1,3);
\draw (1,2) -- (1,3);
\draw (-1,2) -- (1,2);
\draw (0,3) -- (0,2.5);

\draw (-1,-1) -- (1,-1);
\draw (-1,-1) -- (-1,-3);
\draw (1,-1) -- (1,-3);

\draw (-2.5,0) node[left]{$p = $};
\end{tikzpicture}
\end{center}

This pictorial description makes it easy to work with \emph{blocks}, which we now define.

\begin{de}
Let $p$ be a partition.
\begin{itemize}
\item A maximal set of points which are all connected (i.e.~one of the subsets defining the partition) is called a \emph{block} of $p$.
\item If moreover this block consists only of neighbouring points, then it is called an \emph{interval}.
\item If $b$ contains both upper and lower points (i.e.~the subset contains an element of $\{1, \dots, k\}$ and an element of $\{k+1, \dots, k+l\}$), then it is called a \emph{through-block}.
\item Otherwise, it is called a \emph{non-through-block}.
\end{itemize}
The total number of through-blocks of the partition $p$ is denoted by $t(p)$.
\end{de}

This work focuses on a special type of partitions which are said to be \emph{non-crossing}.

\begin{de}\label{de:noncrossing}
Let $p$ be a partition. A \emph{crossing} in $p$ is a tuple $k_{1} < k_{2} < k_{3} < k_{4}$ of integers such that :
\begin{itemize}
\item $k_{1}$ and $k_{3}$ are in the same block,
\item $k_{2}$ and $k_{4}$ are in the same block,
\item the four points are \emph{not} in the same block.
\end{itemize}
If there is no crossing in $p$, then it is said to be  a \emph{non-crossing} partition. The set of non-crossing partitions will be denoted by $NC$.
\end{de}

The example given above is a non-crossing partition. We now add some further structure on the partitions by introducing colours.

\begin{de}
A \emph{colour set} is a set $\A$ together with an involution denoted by $x\mapsto x^{-1}$. An $\A$-coloured partition is a partition together with an element of $\A$ attached to each point. A coloured partition is said to be non-crossing if the underlying uncoloured partition is non-crossing. The set of $\A$-coloured non-crossing partitions will be denoted by $NC^{\A}$.
\end{de}

Let $p$ be an $\A$-coloured partition. Reading from left to right, we can associate to the upper row of $p$ a word $w$ on $\A$ and to its lower row (again reading from left to right) a word $w'$ on $\A$. For a set of partitions $\CC$, we will denote by $\CC(w, w')$ the subset of all partitions in $\CC$ such that the upper row is coloured by $w$ and the lower row is coloured by $w'$ and we will denote by $\vert w\vert$ the length of a word $w$. There are several fundamental operations available on partitions called the \emph{category operations} :
\begin{itemize}
\item If $p\in \CC(w, w')$ and $q\in \CC(z, z')$, then $p\otimes q\in \CC (w.z, w'.z')$ is their \emph{horizontal concatenation}, i.e.~the first $\vert w\vert$ of the $\vert w\vert + \vert z\vert$ upper points are connected by $p$ to the first $\vert w'\vert$ of the $\vert w'\vert + \vert z'\vert$ lower points, whereas $q$ connects the remaining $\vert z\vert$ upper points with the remaining $\vert z'\vert$ lower points.
\item If $p\in \CC(w, w')$ and $q\in \CC(w', w'')$, then $qp\in \CC(w, w'')$ is their \emph{vertical concatenation}, i.e.~$\vert w\vert$ upper points are connected by $p$ to $\vert w'\vert$ middle points and the lines are then continued by $q$ to $\vert w''\vert$ lower points. This process may produce loops in the partition. More precisely, consider the set $L$ of elements in $\{1, \dots, \vert w'\vert\}$ which are not connected to an upper point of $p$ nor to a lower point of $q$. The lower row of $p$ and the upper row of $q$ both induce a partition of the set $L$. The maximum (with respect to inclusion) of these two partitions is the \emph{loop partition} of $L$, its blocks are called \emph{loops} and their number is denoted by $\rl(q, p)$. To complete the operation, we remove all the loops. Note that we can only perform this vertical concatenation if the words associated to the lower row of $p$ and the upper row of $q$ match.
\item If $p\in \CC(w, w')$, then $p^{*}\in \CC(w', w)$ is the partition obtained by reflecting $p$ with respect to an horizontal axis between the two rows (without changing the colours).
\item If $w = w_{1}\dots w_{n}$, $w' = w'_{1}\dots w'_{k}$ and $p\in \CC(w, w')$, then rotating the extreme left point of the lower row of $p$ to the extreme left of the upper row and changing its colour to its inverse yields a partition $q\in \CC((w'_{1})^{-1}w_{1}\dots w_{n}, w'_{2}\dots w'_{k})$. The partition $q$ is called a \emph{rotated version} of $p$. One can also perform rotations on the right and from the upper to the lower row.
\end{itemize}

Let us say that for an element $x\in \A$, the \emph{$x$-identity partition} is the partition $\idpart\in \CC(x, x)$ coloured with $x$ on both ends. We are now ready for the definition of a category of coloured partitions, the fundamental object of this work.

\begin{de}
A \emph{category of $\A$-coloured partitions} $\CC$ is the data of a set of $\A$-coloured partitions $\CC(w, w')$ for all words $w$ and $w'$ on $\A$, which is stable under all the category operations and contains the $x$-identity partition for all $x\in \A$.
\end{de}

Using the category operations one can define another operation on partitions which will play an important role in the sequel. If $w = w_{1}\dots w_{n}$ is a word on $\A$, we set $\overline{w} = \overline{w}_{n}\dots \overline{w}_{1}$.

\begin{de}
The \emph{conjugate} of a partition $p\in NC^{\A}(w, w')$ is the partition $\overline{p}\in NC^{\A}(\overline{w}', \overline{w})$ obtained by rotating $p$ upside down. Note that categories of partitions are by definition stable under taking conjugates.
\end{de}

In the next subsection, we will explain the link between categories of partitions and compact quantum groups. This link relies on the following way of associating linear maps to partitions :

\begin{de}
Let $N$ be an integer and let $(e_{1}, \dots, e_{N})$ be the canonical basis of $\C^{N}$. For any partition $p$, we define a linear map
\begin{equation*}
T_{p}:(\C^{N})^{\otimes k} \mapsto (\C^{N})^{\otimes l}
\end{equation*}
by the following formula :
\begin{equation*}
T_{p}(e_{i_{1}} \otimes \dots \otimes e_{i_{k}}) = \sum_{j_{1}, \dots, j_{l} = 1}^{n} \delta_{p}(i, j)e_{j_{1}} \otimes \dots \otimes e_{j_{l}},
\end{equation*}
where $\delta_{p}(i, j) = 1$ if and only if all the strings of the partition $p$ connect equal indices of the multi-index $i = (i_{1}, \dots, i_{k})$ in the upper row with equal indices of the multi-index $j = (j_{1}, \dots, j_{l})$ in the lower row. Otherwise, $\delta_{p}(i, j) = 0$.
\end{de}

Note that the colours do not play any role in this definition. The reason why non-crossing partition quantum groups are easier to handle than general partition quantum groups is the following well-known fact (see e.g. \cite[Lem 4.16]{freslon2013representation} for a proof). This result is also responsible for the restriction $N\geqslant 4$ in most of the statements of the present work.

\begin{prop}
Let $N\geqslant 4$ be an integer and let $w$ and $w'$ be words on $\A$. Then, the linear maps $(T_{p})_{p\in NC^{\A}(w, w')}$ are linearly independant.
\end{prop}

\subsection{Partition quantum groups}

Partition quantum groups were introduced in \cite{freslon2014partition} as a generalization of the easy quantum groups defined by T. Banica and R. Speicher in \cite{banica2009liberation}. They fit into the general setting of compact quantum groups of S.L. Woronowicz. We therefore first recall some basic definitions and results of this theory. The reader may refer for instance to the book \cite{neshveyev2014compact} for details and proofs.

\begin{de}
A \emph{compact quantum group} is a pair $\G = (C(\G), \D)$ where $C(\G)$ is a unital C*-algebra and
\begin{equation*}
\D: C(\G)\rightarrow C(\G)\otimes C(\G)
\end{equation*}
is a unital $*$-homomorphism such that $(\D\otimes \ii)\circ\D = (\ii \otimes \D)\circ \D$ and the linear spans of $\D(C(\G))(1\otimes C(\G))$ and $\D(C(\G))(C(\G)\otimes 1)$ are dense in $C(\G)\otimes C(\G)$ (all the tensor products of C*-algebras are spatial).
\end{de}

The reader should be careful that the notation $C(\G)$ is symbolic in the sense that there is no topological space $\G$. In particular, the C*-algebra $C(\G)$ can (and will in the cases we study) be noncommutative. The fundamental notion for our purpose is that of a finite-dimensional representation.

\begin{de}
Let $\G$ be a compact quantum group. A \emph{representation} of $\G$ of dimension $n$ is a matrix
\begin{equation*}
(u_{ij})_{1\leqslant i, j\leqslant n}\in M_{n}(C(\G)) \simeq C(\G) \otimes M_{n}(\C)
\end{equation*}
such that
\begin{equation*}
\D(u_{ij}) = \displaystyle\sum_{k=1}^{n} u_{ik}\otimes u_{kj}
\end{equation*}
for every $1 \leqslant i, j \leqslant n$. Its \emph{contragredient} representation $\co{u}$ is defined by $\overline{u}_{ij} = u_{ij}^{*}$. A representation $u$ is said to be \emph{unitary} if it is a unitary element of $M_{n}(C(\G))$. The \emph{trivial representation} of $\G$ is $\varepsilon = 1\otimes 1\in C(\G)\otimes \C$.
\end{de}

An \emph{intertwiner} between two representations $u$ and $v$ of dimension respectively $n$ and $m$ is a linear map $T: \C^{n} \rightarrow \C^{m}$ such that
\begin{equation*}
(\ii \otimes T)\circ u = v\circ (\ii \otimes T).
\end{equation*}
The set of intertwiners between $u$ and $v$ is denoted by $\Mor_{\G}(u, v)$, or simply $\Mor(u, v)$ if there is no ambiguity. If there exists a unitary intertwiner between $u$ and $v$, then they are said to be \emph{unitarily equivalent}. A representation $u$ is said to be \emph{irreducible} if $\Mor(u, u) = \C.\Id$. The \emph{tensor product} of two representations $u$ and $v$ is the representation
\begin{equation*}
u\otimes v = u_{12}v_{13}\in C(\G)\otimes M_{n}(\C)\otimes M_{m}(\C) \simeq C(\G)\otimes M_{nm}(\C),
\end{equation*}
where we used the \emph{leg-numbering} notations : for an operator $X$ acting on a twofold tensor product, $X_{ij}$ is the extension of $X$ acting on the $i$-th and $j$-th tensors of a triple tensor product. Compact quantum groups have a tractable representation theory because of the following fundamental result :

\begin{thm}[Woronowicz]\label{thm:peterweyl}
Every unitary representation of a compact quantum group is unitarily equivalent to a direct sum of irreducible unitary representations. Moreover, any irreducible representation is finite-dimensional.
\end{thm}

This theorem implies that finite-dimensional representations contain all the information about a compact quantum group $\G$. It should therefore be possible to recover $\G$ from its representation theory. Taking a more categorical point of view, this is the content of S.L. Woronowicz's Tannaka-Krein theorem proved in \cite{woronowicz1988tannaka}. We will not give a general statement here, but simply apply it in our setting. To do this, we need to introduce some notations. Let $\G$ be a compact quantum group and let $(u^{x})_{x\in \A}$ be a family of finite-dimensional representations of $\G$, each acting on a finite-dimensional Hilbert space $V^{x}$. If $w$ is a word on $\A$, then we set $u^{\otimes w} = u^{w_{1}}\otimes \dots \otimes u^{w_{n}}$, which is a representation acting on the Hilbert space $V^{\otimes w} = V^{w_{1}}\otimes \dots \otimes V^{w_{n}}$. We refer the reader to \cite[Thm 3.2.8]{freslon2014partition} for a proof of the next result, which is the starting point of the theory of partition quantum groups.

\begin{thm}\label{thm:tannakakrein}
Let $\CC$ be a category of $\A$-coloured partitions and let $N$ be an integer. Then, there exists a unique (up to isomorphism) compact quantum group $\G$ together with representations $(u^{x})_{x\in \A}$ such that
\begin{itemize}
\item Any representation of $\G$ is equivalent to a subrepresentation of a direct sum of tensor products of the representations $u^{x}$.
\item For any words $w$ and $w'$ on $\A$, $\Mor_{\G}(u^{\otimes w}, u^{\otimes w'}) = \Span\{T_{p}, p\in \CC(w, w')\}$.
\end{itemize}
\end{thm}

\begin{de}
Let $\CC$ be a category of partitions and let $N$ be an integer. The compact quantum group given by Theorem \ref{thm:tannakakrein} will be denoted by $\G_{N}(\CC)$ and called the \emph{partition quantum group} associated to $\CC$ and $N$.
\end{de}

\subsection{Projective partitions and representations}\label{subsec:projective}

Let $\CC$ be a category of partitions and let $N$ be an integer. The compact quantum group $\G_{N}(\CC)$ is defined through its category of representations, which is itself completely determined by $\CC$. Hence, it should be possible to describe the representations of $\G$ solely in terms of the partitions in $\CC$. This was done in a joint work with M. Weber \cite{freslon2013representation} (it is restricted to the one-colour case but generalizes straightforwardly to the general case, see \cite{freslon2014partition}). A crucial role is played by the so-called \emph{projective partitions}, which we now define.

\begin{de}\label{de:projective}
A partition $p\in P(w, w)$ is said to be \emph{projective} if $pp = p = p^{*}$. Moreover,
\begin{itemize}
\item A projective partition $p$ is said to be \emph{dominated} by $q$ if $qp = p$. Then, $pq = p$ and we write $p\preceq q$.
\item Two projective partitions $p$ and $q$ are said to be \emph{equivalent} in a category of partitions $\CC$ if there exists $r\in \CC$ such that $r^{*}r = p$ and $rr^{*} = q$. We then write $p\sim q$ or $p\sim_{\CC} q$ if we want to keep track of the category of partitions.
\end{itemize}
Note that if $p\sim q$, then $t(p) = t(q)$.
\end{de}

\begin{rem}\label{rem:noncrossingdecomposition}
A projective noncrossing partition decomposes always as horizontal concatenations of through-block projective partitions with exactly one through-block and non-through-block projective partitions with endpoints connected. This decomposition will be used repeatedly throughout this work.
\end{rem}

The order relation $\preceq$ is not to be confused with the usual order relation $\leqslant$ (sometimes called "being coarser") on partitions. Note also that if $p$ and $q$ are equivalent in $\CC$, then they in fact both belong to $\CC$. Moreover, the equivalence is implemented by a unique partition denoted by $r_{p_{2}}^{p_{1}}$. Here is a kind of converse to this statement, proved in \cite[Prop 2.18]{freslon2013representation} :

\begin{prop}
For any partition $r$, the partitions $r^{*}r$ and $rr^{*}$ are both projective partitions (and $r$ is an equivalence between them).
\end{prop}

It was shown in \cite{freslon2013representation} that one can associate to any projective partition $p\in \CC(w, w)$ a subrepresentation $u_{p}$ of $u^{\otimes w}$. The study of these representations can be very complicated in general but in the noncrossing case things become more tractable.

\begin{thm}\label{thm:partitionrepresentations}
Let $N\geqslant 4$ be an integer and let $\CC$ be a category of noncrossing partitions. Then,
\begin{itemize}
\item $u_{p}$ is a non-zero irreducible representation for any projective partition $p$,
\item $u_{p}\sim u_{q}$ if and only if $p\sim q$,
\item $u^{\otimes w} = \displaystyle\bigoplus_{p\in \Proj(w)} u_{p}$ as a direct sum of representations, where $\Proj(w)$ denotes the set of projective partitions in $\CC(w, w)$. In particular, any irreducible representation is equivalent to $u_{p}$ for some projective partition $p$.
\end{itemize}
\end{thm}

\section{General results on noncrossing partition quantum groups}\label{sec:structure}

The aim of this section is to give some general results on the class of compact quantum groups associated to categories of noncrossing partitions. As will appear in this work, a complete classification seems out of reach for the moment. However, one can at least describe some constructions which preserve this structure. One way of completing a classification in that direction would then be to prove that any noncrossing partition quantum group can be obtained by applying these constructions to a certain elementary class (see Section \ref{sec:summary} for details).

We first fix some shorthand notations. If $w$ and $w'$ are words on $\A$, we denote by $\pi(w, w')$ the unique one-block partition in $NC^{\A}(w, w')$. We also denote by $\beta(w, w')$ the partition in $NC^{\A}(w, w')$ having exactly one upper block and one lower block. These partitions can be pictorially represented as follows :

\begin{center}
\begin{tikzpicture}[scale=0.5]
\draw (-2.5,0.5) -- (2.5,0.5);
\draw (-2.5,-0.5) -- (2.5,-0.5);

\draw (-2.5,0.5) -- (-2.5,1.5);
\draw (2.5,0.5) -- (2.5,1.5);
\draw (-2.5,-0.5) -- (-2.5,-1.5);
\draw (2.5,-0.5) -- (2.5,-1.5);

\draw (-1,0.5) -- (-1,1.5);
\draw (1,0.5) -- (1,1.5);
\draw (-1,-0.5) -- (-1,-1.5);
\draw (1,-0.5) -- (1,-1.5);

\draw (0,0.5) -- (0,-0.5);

\draw (0,1.5) node[below]{$\dots$};
\draw (0,-1.5) node[above]{$\dots$};

\draw (-2.5,1.5) node[above]{$w_{1}$};
\draw (-1,1.5) node[above]{$w_{2}$};
\draw (1,1.5) node[above]{$w_{n-1}$};
\draw (2.5,1.5) node[above]{$w_{n}$};

\draw (-2.5,-1.5) node[below]{$w'_{1}$};
\draw (-1,-1.5) node[below]{$w'_{2}$};
\draw (1,-1.5) node[below]{$w'_{k-1}$};
\draw (2.5,-1.5) node[below]{$w'_{k}$};

\draw (-3,0) node[left]{$\pi(w, w') = $};
\end{tikzpicture}
\begin{tikzpicture}[scale=0.5]
\draw (-2.5,0.5) -- (2.5,0.5);
\draw (-2.5,-0.5) -- (2.5,-0.5);

\draw (-2.5,0.5) -- (-2.5,1.5);
\draw (2.5,0.5) -- (2.5,1.5);
\draw (-2.5,-0.5) -- (-2.5,-1.5);
\draw (2.5,-0.5) -- (2.5,-1.5);

\draw (-1,0.5) -- (-1,1.5);
\draw (1,0.5) -- (1,1.5);
\draw (-1,-0.5) -- (-1,-1.5);
\draw (1,-0.5) -- (1,-1.5);

\draw (0,1.5) node[below]{$\dots$};
\draw (0,-1.5) node[above]{$\dots$};

\draw (-2.5,1.5) node[above]{$w_{1}$};
\draw (-1,1.5) node[above]{$w_{2}$};
\draw (1,1.5) node[above]{$w_{n-1}$};
\draw (2.5,1.5) node[above]{$w_{n}$};

\draw (-2.5,-1.5) node[below]{$w'_{1}$};
\draw (-1,-1.5) node[below]{$w'_{2}$};
\draw (1,-1.5) node[below]{$w'_{k-1}$};
\draw (2.5,-1.5) node[below]{$w'_{k}$};

\draw (-3,0) node[left]{$\beta(w, w') = $};
\end{tikzpicture}
\end{center}
Note that for $x\in \A$, $\pi(x, x)$ is simply the $x$-identity partition.

\subsection{One-dimensional representations}\label{subsec:oned}

In this subsection we study one-dimensional representations, which are the crucial point to understand noncrossing partition quantum groups. 

\subsubsection{Structure}

In \cite{freslon2014partition}, we classified a class of noncrossing partition quantum groups called free quantum groups, which can be defined as follows :

\begin{de}\label{de:free}
A partition quantum group is \emph{free} if it has no nontrivial one-dimensional representation.
\end{de}

Thus, the rest of the classification must deal with one-dimensional representations. At the level of categories of partitions, this can be easily translated (see \cite[Thm 4.18]{freslon2013fusion} for details).

\begin{thm}
Let $\CC$ be a category of noncrossing partitions and let $N\geqslant 4$ be an integer. Then, $\G_{N}(\CC)$ is free if and only if $\CC$ is \emph{block-stable}, i.e.~for any partition $p\in \CC$ and any block $b$ of $p$, we have $b\in \CC$.
\end{thm}

Dealing with categories of partitions which are not block-stable is difficult in general and one aim of this subsection is to develop some tools to describe the effects of this lack of block-stability. We start with some elementary properties.

\begin{lem}\label{lem:nonthroughblockstructure}
Let $\CC$ be a noncrossing category of partitions and let $p$ be a projective partition. Then, $u_{p}$ is one-dimensional if and only if $t(p)=0$. Moreover, $p$ is then equal to $b^{*}b$ for some partition $b$ lying on one line.
\end{lem}

\begin{proof}
The fact that $\dim(u_{p}) = 1$ if and only if $t(p) = 0$ was proved in \cite[Lem 5.1]{freslon2013fusion}. Under this assumption, let $b$ be the upper row of $p$. Then, $b^{*}b = p^{*}p = p$.
\end{proof}

Thus, we will focus on non-through-block projective partitions. There is a useful way to produce such partitions using the following notion :

\begin{de}
Let $p$ be a partition of $\{1, \cdots, k\}$. A subpartition $q\subset p$ is said to be \emph{full} if it is a partition of a subset of the form $\{a, a+1, \cdots, a+b\}$ for some $1\leqslant a\leqslant a+b\leqslant k$.
\end{de}

Pictorially, this definition means that we can rotate $p$ on one line so that between any two points of $q$, all the points are also in $q$. The idea is that a full subpartition can be "isolated" from the all the other points and this makes it possible to extract it as a projective partition without through-blocks. Note that the main example of full subpartitions are intervals.

\begin{prop}\label{prop:intervalblockstable}
Let $p\in \CC$ and let $q\subset p$ be a \emph{full} subpartition in $p$ lying on one line. Then, $q^{*}q\in \CC$.
\end{prop}

\begin{proof}
By definition, we can rotate $p$ so that the upper points are exactly the points of $q$. Denoting by $p'$ this partition, we get $q^{*}q = p^{\prime*}p'\in \CC$.
\end{proof}

Proposition \ref{prop:intervalblockstable} gives us a way of producing one-dimensional representations out of arbitrary partitions. We then need to be able to compare them with respect to the equivalence relation on projective partitions. If they have the same colouring, then this is easy.

\begin{lem}\label{lem:equivalence1d}
Let $p, q\in \CC(w, w)$ be projective partitions with $t(p) = t(q) = 0$. Then, $p\sim q$.
\end{lem}

\begin{proof}
Because $p$ and $q$ have the same colouring, the composition $r = pq$ makes sense and belongs to $\CC$. But since the partitions have no through-blocks, $r$ is an equivalence between $p$ and $q$.
\end{proof}

Let us end by recalling \cite[Lem 4.17]{freslon2013fusion} which characterizes the triviality of $u_{p}$ :

\begin{lem}
Let $N\geqslant 4$, let $\CC$ be a category of noncrossing partitions and let $p$ be a non-through-block projective partition. Then, $u_{p}$ is trivial if and only if the upper row of $p$ belongs to $\CC$.
\end{lem}

\subsubsection{Group-like and commutation relations}

Using what precedes, we can start giving some general constructions preserving the partition structure of compact quantum groups. By definition, a one-dimensional representation is a unitary element $t\in C(\G)$ such that $\D(t) = t\otimes t$. Such elements are often called \emph{group-like} and they form a group, denoted by $\mathcal{G}(\G)$. Before going further, let us summarize the descriptions of one-dimensional representations that we have so far :
\begin{itemize}
\item group-like elements, which are the same as one-dimensional representations,
\item non-through-block projective partitions, which are the same as partitions of the form $b^{*}b$ for a partition $b$ lying on one line.
\end{itemize}
Adding relations in $\mathcal{G}(\G)$ preserves the partition structure, where by adding relations we mean quotienting by the Hopf $*$-ideal generated by some relations (this is the same as taking the largest quotient which is a compact quantum group and satisfies the given relations). This is our first stability result.

\begin{prop}\label{prop:1dstability}
Let $N\geqslant 4$ be an integer, let $\CC$ be a category of noncrossing partitions and let $t$ be a group-like element in $C(\G_{N}(\CC))$. Then, the quotient of $C(\G_{N}(\CC))$ by the relation $t = 1$ is a noncrossing partition quantum group.
\end{prop}

\begin{proof}
By Lemma \ref{lem:nonthroughblockstructure}, there exists a non-through-block projective partition $p = b^{*}b$ such that $u_{p}$ is equivalent to $t$. Seen as group-like elements, two one-dimensional representations are equivalent if and only if they are equal. Thus, adding the relation $t = 1$ is the same as adding the relation $u_{p} = 1$, which in turn amounts to adding $b$ to the category of partitions.
\end{proof}

Since this kind of quotient will appear in every step of the classification, we give it a name.

\begin{de}\label{de:grouplike}
Let $\G$ and $\HH$ be compact quantum groups. We say that $C(\HH)$ is a \emph{quotient of $C(\G)$ by group-like relations} if it is obtained by making some group-like elements trivial. 
\end{de}

Here is another construction with one-dimensional representations involving commutation relations :

\begin{prop}\label{prop:commutationrelations}
Let $N\geqslant 4$ be an integer, let $\CC$ be a category of noncrossing partitions, let $t$ be a group-like element in $C(\G_{N}(\CC))$ and let $u$ be an irreducible representation of $\G_{N}(\CC)$. Then, the quotient of $C(\G_{N}(\CC))$ by the relations
\begin{equation*}
t u_{ij} = u_{ij}t
\end{equation*}
for all $1\leqslant i, j\leqslant \dim(u)$ is a noncrossing partition quantum group.
\end{prop}

\begin{proof}
Let $p$ be a projective partition such that $u_{p}$ is equivalent to $u$ and let $q$ be a projective partition such that $u_{q} = t$ as group-like elements (equivalent one-dimensional representations yield the same group-like element). By Lemma \ref{lem:nonthroughblockstructure}, $q = b^{*}b$ for some $b\in \CC(w, \emptyset)$ and we can consider the partition $r = b\otimes p\otimes \overline{b}^{*}$. The proof relies on the constructions of \cite{freslon2013representation} but we will make it as self-contained as possible.

If $p$ has $2n_{2}$ points, then there exists a scalar $\lambda_{p}\in \C$ such that $\lambda_{p}(\id\otimes T_{p})u^{\otimes n_{2}}(\id\otimes T_{p})$ is a representations which will be denoted by $U_{p}$. If moreover $b$ has $n_{1}$ points, then for $x_{2}\in (\C^{N})^{\otimes n_{2}}$ and $x_{1}, x_{3}\in (\C^{N})^{\otimes n_{1}}$, we have
\begin{align*}
u^{\otimes n_{2}}(\id\otimes T_{r})(1\otimes x_{1}\otimes x_{2}\otimes x_{3}) & = N^{-\mathrm{rl}(p, p)}T_{b}(x_{1})T_{\overline{b}^{*}}(x_{2})u^{\otimes n_{2}}(1\otimes T_{p}(x_{2})) \\
& = N^{-\mathrm{rl}(p, p)}\lambda_{p}^{-1}T_{b}(x_{1})T_{\overline{b}^{*}}(x_{2}) U_{p}(1\otimes x_{2})
\end{align*}
using the fact that $T_{p} = N^{-\mathrm{rl}(p, p)}T_{p}T_{p}$. On the other hand, if $b$ has $\ell$ blocks then
\begin{equation*}
(\id\otimes T_{b^{*}b})u^{\otimes n_{1}}(\id\otimes T_{b^{*}b}) = N^{2\ell}(u_{b^{*}b}\otimes 1) = N^{2\ell}(t\otimes 1),
\end{equation*}
so that
\begin{align*}
(\id\otimes T_{b})u^{\otimes n_{1}}(1\otimes x_{i}) & = N^{-\ell}(\id\otimes T_{b})(\id\otimes T_{b^{*}b})u^{\otimes n_{1}}(1\otimes x) \\
& = N^{-2\ell}(\id\otimes T_{b})(\id\otimes T_{b^{*}b})u^{\otimes n_{1}}(1\otimes T_{b^{*}b}(x)) \\
& = (\id\otimes T_{b})(t\otimes x) \\
& = T(b)(x)t.
\end{align*}
The same holds for $\overline{b}^{*}$ with $t$ replaced by $t^{*} = t^{-1}$, yielding
\begin{align*}
(\id\otimes T_{r})u^{\otimes (2n_{1}+n_{2})}(1\otimes x_{1}\otimes x_{2}\otimes x_{3}) & = N^{-\mathrm{rl}(p, p)}\lambda_{p}^{-1}(\id\otimes T_{b})(u^{\otimes n_{1}}(1\otimes x_{1}))U_{p}(1\otimes x_{2}) \\
& \times (\id\otimes T_{\overline{b}^{*}})(u^{\otimes n_{1}}(1\otimes x_{3})) \\
& = (t\otimes \id)U_{p}(t^{*}\otimes \id)(1\otimes x_{2}).
\end{align*}
Thus, $(t\otimes \id)U_{p}(t^{*}\otimes \id) = U_{p}$. Now, it follows from \cite[Rem 4.3]{freslon2013representation} and \cite[Prop 3.7]{freslon2013fusion} that $U_{p}$ splits as the direct sum of all the representations $u_{p'}$ for $p'\leqslant p$. Since by concatenating we have $b\otimes p'\otimes \overline{b}^{*}\in \CC$, a straightforward induction then shows that $(t\otimes \id)u_{p'}(t^{*}\otimes \id) = u_{p'}$ for all $p'\leqslant p$, hence in particular for $p' = p$. This means that adding $r$ to the category of partitions is the same as adding the relations $t (u_{p})_{ij} = (u_{p})_{ij}t$ for all $1\leqslant i, j\leqslant \dim(u_{p})$.

Consider now a unitary matrix $V\in M_{\dim(u)}(\C.1)\subset M_{\dim(u)}(C(\G))$ such that $Vu_{p}V^{*} = u$. If $T = t.\Id$, then the commutation relation of the previous paragraph can be written as $Tu_{p} = u_{p}T$. Since $T$ commutes with $M_{\dim(u)}(\C.1)$, $Tv = VT$ so that $Tu = uT$, which in turn translates into the commutation relation in the statement.
\end{proof}

\subsubsection{Twisted amalgamation}

One can also use one-dimensional representations to "twist" an amalgamated free product. To explain this, let us first recall the usual construction of free products : given two compact quantum groups $\G_{1}$ and $\G_{2}$, the free product C*-algebra $C(\G_{1})\ast C(\G_{2})$ can be turned into a compact quantum group called their \emph{free product} and denoted by $\G_{1}\ast \G_{2}$. If the quantum groups come from partitions, then so does their free product, and the associated category of partitions can be explicitly described.

\begin{de}
Let $\CC_{1}$ and $\CC_{2}$ be two categories of partitions coloured by $\A_{1}$ and $\A_{2}$ respectively. Their \emph{free product} is the category of partitions $\CC = \CC_{1}\ast\CC_{2}$ generated in $NC^{\A_{1}\sqcup \A_{2}}$ by $\CC_{1}$ and $\CC_{2}$.
\end{de}

According to \cite[Prop 4.12]{freslon2014partition} $\G_{N}(\CC_{1})\ast\G_{N}(\CC_{2}) = C(\G_{N}(\CC_{1}\ast\CC_{2}))$ which justifies the terminology. Assume now that $\G_{1}$ and $\G_{2}$ are two compact quantum groups with a \emph{common dual quantum subgroup} $\HH$ in the sense that there are injective $*$-homomorphisms $i_{1} : C(\HH)\rightarrow C(\G_{1})$ and $i_{2} : C(\HH)\rightarrow C(\G_{2})$ preserving the coproducts. The object $\HH$ can be interpreted as the dual of a common discrete quantum subgroup of the duals of $\G_{1}$ and $\G_{2}$, hence the name. The quotient of $C(\G_{1})\ast C(\G_{2})$ by the ideal generated by elements of the form $i_{1}(x)-i_{2}(x)$ is a new C*-algebra called an \emph{amalgamated free product} and denoted by $C(\G_{1})\ast_{C(\HH)}C(\G_{2})$. It was proved by S. Wang in \cite{wang1995free} that there is a natural compact quantum group structure on this C*-algebra. The resulting object is called the \emph{free product of $\G_{1}$ and $\G_{2}$ amalgamated over $\HH$} and denoted by $\G_{1}\ast_{\HH}\G_{2}$. If now $t$ is a group-like element, we may conjugate one of the copies of $\HH$ by it before identifying it with the other copy, yielding the following more general notion of amalgamation :

\begin{de}\label{de:twistedamalgamtion}
Let $\G_{1}$ and $\G_{2}$ be compact quantum groups, let $\HH$ be a common dual quantum subgroup and let $K$ be a set of group-like elements in $C(\G_{1}\ast\G_{2})$. The \emph{amalgamated free product over $\HH$ twisted by $K$} is the quantum group obtained by quotienting $C(\G_{1}\ast\G_{2})$ by the relations
\begin{equation*}
t i_{1}(x)t^{-1} = i_{2}(x)
\end{equation*}
for all $x\in C(\HH)$ and all $t\in K$.
\end{de}

As for group-like relations, "quotienting by relations" here means taking the largest quotient compact quantum group whose associated C*-algebra satisfies these relations. This construction is interesting in our setting due to the following :

\begin{prop}\label{prop:amalgamatedstability}
Let $N\geqslant 4$ be an integer, let $\CC_{1}$ and $\CC_{2}$ be two categories of noncrossing partitions and let $\HH$ be a common dual quantum subgroup (which is not assumed to be a partition quantum group) of $\G_{N}(\CC_{1})$ and $\G_{N}(\CC_{2})$. Then, for any set of group-like elements $K$, the amalgamated free product over $\HH$ twisted by $K$ is a partition quantum group.
\end{prop}

\begin{proof}
First note that because the coefficients of irreducible representations generate a dense subalgebra of $C(\HH)$, it is enough to prove that, for any irreducible representation $\alpha$ of $\HH$ and any $t\in K$, we may add to $\CC_{1}\ast\CC_{2}$ partitions realizing the relations
\begin{equation*}
t i_{1}(u^{\alpha}_{ij})t^{-1} = i_{2}(u^{\alpha}_{ij})
\end{equation*}
for all $1\leqslant i, j\leqslant \dim(\alpha)$ can be obtained by adding partitions. So let $\alpha$ be an irreducible representation of $\HH$, let $p_{1}$ be a projective partition in $\CC_{1}$ such that $u_{p_{1}}$ is equivalent $u' = (i_{1}\otimes \ii)(u^{\alpha})$ and let $p_{2}$ be a projective partition in $\CC_{2}$ such that $u_{p_{2}}$ is equivalent to $u'' = (i_{2}\otimes \ii)(u^{\alpha})$. Let also $q = b^{*}b$ be a non-through-block projective partition such that $u_{q} = t$ as group-like element. The colours of $p_{1}$ and $p_{2}$ do not match, but we can still build a partition out of them using the following construction, which is a particular case of the through-block decomposition of \cite[Prop 2.9]{freslon2013representation} :
\begin{itemize}
\item we cut $p_{1}$ in the middle, keeping its upper row together with $t(p_{1})$ lower points which we colour by some fixed element $x\in\A$, and denote the resulting partition by $p_{1}'$.
\item we cut $p_{2}$ in the middle, keeping its lower row together with $t(p_{2})$ upper points which we also colour by $x$, and denote the resulting partition by $p_{2}'$.
\item because $\dim(u') = \dim(u'')$, $t(p_{1}) = t(p_{2})$ by \cite[Prop 2.16]{freslon2013representation} so that the composition $p_{2}'p_{1}'$ makes sense. We denote it by $r$.
\end{itemize}
The same argument as in the proof of Proposition \ref{prop:commutationrelations} then shows that quotienting by the relations $t u'_{ij}t^{-1} = u''_{ij}$ for all $1\leqslant i, j\leqslant \dim(u)$ is the same as adding the partition $s=b\otimes r\otimes b^{*}$ to the category of partitions. Because the relations $t u'_{ij}t^{-1} = u''_{ij}$ can also be written $t i_{1}(u_{ij})t^{-1} = i_{2}(u_{ij})$, doing this for all irreducible representations of $\HH$ and all $t\in K$ we get the twisted amalgamated free product.
\end{proof}

Note that the fusion rules of these compact quantum groups can be computed using the results of \cite{freslon2013representation}. This is in sharp contrast with the fact that there is no general method to compute the representation theory of an amalgamated free product of compact quantum groups.

\begin{rem}
If two elements $t_{1}$ and $t_{2}$ satisfy the twisting relation above, then $t_{1}^{-1}t_{2}$ commutes with $i_{1}(C(\HH))$. This means that the group-like elements satisfying the twisted amalgamation relation form a coset with respect to the subgroup of group-like elements commuting with $i_{1}(C(\HH))$. In the sequel, we will directly consider the amalgamated free product twisted by a coset.
\end{rem}

\subsection{Free wreath products of pairs}\label{subsec:wreathpairs}

Starting from free quantum groups and applying the constructions above repeatedly yields a large class of noncrossing partition quantum groups, but not all of them. That is the reason why we now introduce another class by slightly extending the free wreath product construction. Recall from \cite{lemeux2013fusion} that given a group $\Gamma$ and a symmetric generating set $S\subset \Gamma$, one can consider the category $\CC_{\Gamma, S}$ of all $S$-coloured partitions such that in each block, the product of the upper colouring equals the product of the lower colouring as elements of $\Gamma$. This is a category of partitions and the associated compact quantum group is the \emph{free wreath product} $\widehat{\Gamma}\wr_{\ast}S_{N}^{+}$ as defined by J. Bichon in \cite{bichon2004free} (it does not depend on the choice of $S$). Note that by definition, $\CC_{\Gamma, S}$ contains the partition $\pi(w, w)$ for any word $w$ on $S$.

Now let $\gamma\in \Gamma$ and choose $g_{1}, \dots, g_{n}\in S$ such that $g_{1}\dots g_{n} = \gamma$. If we add to $\CC_{\Gamma, S}$ the partition
\begin{equation*}
\beta_{\gamma} = \beta(g_{1}\dots g_{n}, g_{1}\dots g_{n}),
\end{equation*}
we produce a new compact quantum group with a non-trivial one-dimensional representation. If $\gamma = g_{1}\cdots g_{n}$ and $\gamma' = g'_{1}\cdots g'_{n}$, then
\begin{equation*}
\pi(g_{1}\cdots g_{n}g'_{1}\cdots g'_{n}, g_{1}\cdots g_{n}g'_{1}\cdots g'_{n})(\beta_{\gamma}\otimes\beta_{\gamma'})\pi(g_{1}\cdots g_{n}g'_{1}\cdots g'_{n}, g_{1}\cdots g_{n}g'_{1}\cdots g'_{n}) = \beta_{\gamma\gamma'}.
\end{equation*}
Since moreover $\overline{\beta}_{\gamma} = \beta_{\gamma^{-1}}$, adding $\beta_{\gamma}$ is the same as adding $\beta_{\lambda}$ for all $\lambda$ in the subgroup of $\Gamma$ generated by $\gamma$, so that our construction only depends on the choice of a subgroup of $\Gamma$. This leads to the following definition :

\begin{de}
Let $\Lambda\subset \Gamma$ be a subgroup. We denote by $\CC_{\Gamma, \Lambda, S}$ the category of partitions generated by $\CC_{\Gamma, S}$ and $\beta_{\lambda}$ for all $\lambda\in \Lambda$. The associated compact quantum group will be denoted by $H_{N}^{+}(\Gamma, \Lambda)$ and called the \emph{free wreath product of the pair $(\Gamma, \Lambda)$}.
\end{de}

\begin{rem}
If $\lambda = h_{1}\cdots h_{\ell}$ is another decomposition of $\lambda$ with respect to the generating set $S$, then because $\pi(g_{1}\cdots g_{n}, h_{1}\cdots h_{\ell})\in \CC_{\Gamma, S}$ by definition, adding $\beta(g_{1}\cdots g_{n}, g_{1}\cdots g_{n})$ is equivalent to adding $\beta(h_{1}\cdots h_{\ell}, h_{1}\cdots h_{\ell})$. This means that our definition only depends on $\Lambda$ and not on the choice of a particular decomposition of each of its elements.
\end{rem}

The quantum groups $H_{N}^{+}(\Gamma, \Lambda)$ will be our main tool for the classification of two-coloured noncrossing partition quantum groups in the hyperoctahedral and symmetric cases. This is because of the next statement, which is the first illustration of the classification technique explained in the introduction.

\begin{prop}\label{prop:classificationwreath}
Let $N\geqslant 4$ be an integer, let $\Gamma$ be a discrete group with a symmetric generating set $S$, let $\Lambda\subset \Gamma$ be a subgroup and let $\CC$ be a category of $S$-coloured noncrossing partitions containing $\CC_{\Gamma, \Lambda, S}$. Then, $\G_{N}(\CC)$ is a free wreath product of a pair.
\end{prop}

\begin{proof}
We need to show that there exists a group $\widetilde{\Gamma}$ and a subgroup $\widetilde{\Lambda}\subset\widetilde{\Gamma}$ such that $\CC = \CC_{\widetilde{\Gamma}, \widetilde{\Lambda}}$. We will split the proof into three steps.\\

\noindent\textbf{Step 1.} We will first prove that for any projective partition $p\in \CC$, there exists $\Lambda\subset \Lambda'$ such that $p\in \CC_{\Gamma, \Lambda', S}\subset \CC$. Let us start with the case of a non-through-block projective partition of the form $p = q^{*}q$ for some $q\in \CC(w, \emptyset)$ such that the endpoints (the first and the last one) of $q$ are connected. We will prove the result by induction on the number of blocks of $q$.
\begin{itemize}
\item If $q$ has one block, assume it is coloured by $g_{1}, \cdots, g_{n}$ and let $\gamma = g_{1}\cdots g_{n}\in \Gamma$, then $q^{*}q = \beta_{\gamma}$ so that the result holds for the subgroup $\Lambda'$ generated by $\Lambda$ and $\gamma$.
\item Consider now a general $q$ and let $w$ be its colouring. The structure of $q$ is as follows : we have a block containing the endpoints and coloured by, say, $h_{1}, \cdots, h_{\ell}$. Between the points coloured with $h_{i}$ and $h_{i+1}$ we have a block $b_{i}$ (which may be empty). Note that $\beta(w, w) = \pi(w, w)(q^{*}q)\pi(w, w)\in \CC$ and let us set
\begin{equation*}
r = \beta(w, w)\left[\pi(h_{1}, h_{1})\otimes b_{1}^{*}b_{1}\otimes\pi(h_{2}, h_{2})\otimes\cdots\otimes b_{\ell-1}^{*}b_{\ell-1}\otimes \pi(h_{\ell}, h_{\ell})\right]
\end{equation*}
By the straightforward generalization of \cite[Lem 4.3]{freslon2013fusion}, $b_{i}^{*}b_{i}\in \CC$ for all $i$. Thus, by induction we can find $\Lambda\subset \Lambda'$ such that $\beta(w, w)\in \CC_{\Gamma, \Lambda', S}\subset\CC$ and $b_{i}^{*}b_{i}\in \CC_{\Gamma, \Lambda', S}\subset \CC$ for all $i$. Then, $r\in \CC_{\Gamma, \Lambda', S}$ and $p = q^{*}q = r^{*}r \in \CC_{\Gamma, \Lambda', S}$ and the proof is complete.
\end{itemize}

Next we consider a one-block partition $\pi(h_{1}\cdots h_{\ell}, h_{1}\cdots h_{\ell})$ with non-through-block partitions $b_{1}, \cdots, b_{\ell-1}$ between the points. If $p$ denotes the whole partition and $w$ is its upper colouring, we know by \cite[Lem 4.2]{freslon2013fusion} and \textbf{Step 1.} that $b_{i}^{*}b_{i}\in \CC_{\Gamma, \Lambda', S}$ for all $1\leqslant i\leqslant \ell-1$. Moreover,
\begin{equation*}
r = \pi(w, w)\left[\pi(h_{1}, h_{1})\otimes b_{1}^{*}b_{1}\otimes\pi(h_{2}, h_{2})\otimes\cdots\otimes b_{\ell-1}^{*}b_{\ell-1}\otimes \pi(h_{\ell}, h_{\ell})\right]\in \CC_{\Gamma, \Lambda', S}
\end{equation*}
satisfies $r^{*}r = p$. Thus, $p\in \CC_{\Gamma, \Lambda', S}$.

Eventually, as mentioned in Remark \ref{rem:noncrossingdecomposition}, any projective partition can be obtained by horizontal concatenation of ones with exactly one through-block and non-through-block ones with endpoints connected. Thus, we have proved that there exists $\Lambda\subset \Lambda'$ such that $\CC_{\Gamma, \Lambda', S}\subset \CC$ and all projective partitions of $\CC$ lie in $\CC_{\Gamma, \Lambda', S}$.\\

\noindent\textbf{Step 2.} Let us fix now $\Lambda'$ as above. Let $q\in \CC$ be a partition lying on one line and let us denote by $\CC'$ the category of partitions generated by $\CC_{\Gamma, \Lambda', S}$ and $q$. If the colouring of $q$ is $g_{1}, \cdots, g_{n}$, we claim that $\CC'$ coincides with the category of partitions $\CC''$ generated by $\CC_{\Gamma, \Lambda', S}$ and $\beta(g_{1} \cdots g_{n}, \emptyset)$. The inclusion $\CC''\subset \CC'$ follows from
\begin{equation*}
\beta(g_{1}\cdots g_{n}, \emptyset) = q\pi(g_{1}\cdots g_{n}, g_{1}\cdots g_{n})\in \CC'.
\end{equation*}
Conversely, $q^{*}q$ is projective and hence belongs to $\CC_{\Gamma, \Lambda', S}$ by \textbf{Step 1.}, thus
\begin{equation*}
q = \beta(g_{1}\cdots g_{n}, \emptyset)(q^{*}q)\in \CC''.
\end{equation*}
It follows from this that $\CC$ is the category of partitions generated by $\CC_{\Gamma, \Lambda', S}$ and $\beta(g_{1}\cdots g_{n}, \emptyset)$ for all $g_{1}, \cdots, g_{n}$'s corresponding to colourings of partitions of $\CC$ once rotated on one line.\\

\noindent\textbf{Step 3.} The last step is to show that $\CC$ is indeed equal to the category of partitions of a free wreath product of a pair. To do this, let us denote by $\Theta$ the set of all elements $\gamma\in \Gamma$ such that there exists $g_{1}, \cdots, g_{n}\in S$ satisfying $\gamma = g_{1}\cdots g_{n}$ and $\beta(g_{1}\cdots g_{n}, \emptyset)\in \CC$. It follows from the following observations :
\begin{itemize}
\item $[\beta(w, \emptyset)\otimes\beta(w', \emptyset)]\pi(w.w', w.w') = \beta(w.w', \emptyset)$,
\item $\overline{\beta(g_{1}\cdots g_{n}, \emptyset)}^{*} = \beta(g_{n}^{-1}\cdots g_{1}^{-1}, \emptyset)$,
\item $\beta(g_{1}\cdots g_{n}, \emptyset)^{*}\beta(g_{1}\cdots g_{n}, \emptyset) = \beta(g_{1}\cdots g_{n}, g_{1}\cdots g_{n})$
\end{itemize}
that $\Theta$ is a subgroup of $\Lambda$. Moreover, if $\gamma\in \Theta$ can be represented by a word $w$ and $\gamma'\in \Gamma$ can be represented by the word $w'$, then rotating $\pi(w', w')\otimes\beta(w, \emptyset)$ on one line we can get a partition with colouring $w'.w.w^{\prime -1}$. Concatenating with $\pi(w'.w.w^{\prime -1}, w'.w.w^{\prime -1})$ then shows that $\beta(w'.w.w^{\prime -1}, \emptyset)\in \CC$. In other words, $\Theta$ is a normal subgroup of $\Gamma$ (hence also of $\Lambda$). Let us denote by $\widetilde{\Gamma}$ and $\widetilde{\Lambda}$ the quotients of $\Gamma$ and $\Lambda$ by $\Theta$. Identifying the elements of $S$ with their images in the quotient, we can see $\CC_{\widetilde{\Gamma}, \widetilde{\Lambda}, S}$ as a category of $S$-coloured partitions and it contains $\CC_{\Gamma, \Lambda, S}$ as well as $\beta(g_{1}\cdots g_{n}, \emptyset)$ as soon as $g_{1}\cdots g_{n}\in \Theta$, i.e.~$\CC\subset\CC_{\widetilde{\Gamma}, \widetilde{\Lambda}, S}$. Conversely, let $\pi(g_{1}\cdots g_{n}, g'_{1}\cdots g'_{m})\in \CC_{\widetilde{\Gamma}, \widetilde{\Lambda}, S}$. This means that $g_{1}\cdots g_{n} = g'_{1}\cdots g'_{m}$ in $\widetilde{\Gamma}$, which in turn translates into the existence of an element $\gamma'' = h_{1}\cdots h_{\ell}\in \Theta$ such that $g_{1}\cdots g_{n} = g'_{1}\cdots g'_{m}h_{1}\cdots h_{\ell}$ in $\Gamma$. At the level of partitions, we then have that
\begin{equation*}
\pi(g_{1}\cdots g_{n}, g'_{1}\cdots g'_{m}) = [\pi(g'_{1}\cdots g'_{m}, g'_{1}\cdots g'_{m})\otimes\beta(h_{1}\cdots h_{\ell}, \emptyset)]\pi(g_{1}\cdots g_{n}, g'_{1}\cdots g'_{m}h_{1}\cdots h_{\ell})\in \CC.
\end{equation*}
A similar argument shows that $\beta(g_{1}\cdots g_{n}, g_{1}\cdots g_{n})\in \CC$ as soon as $g_{1}\cdots g_{n}\in \widetilde{\Lambda}$, completing the proof.
\end{proof}

\section{Classification I : pair partitions}\label{sec:orthogonal}

We will now start the classification of noncrossing partition quantum groups on a set of two colours $\A = \{x, y\}$. If $x^{-1} = y$, then we are considering the \emph{unitary easy quantum groups} which have been completely classified by P. Tarrago and M. Weber in \cite{tarrago2015unitary}. We can therefore assume that $x^{-1} = x$ and $y^{-1} = y$. We will consider a category of $\A$-coloured noncrossing partitions $\CC$. Note that if $\CC$ contains the partition $\pi(x, y)$, then by definition the identity map on $\C^{N}$ is an equivalence between $u^{x}$ and $u^{y}$. In other words, they are equal and the quantum group $\G_{N}(\CC)$ is an orthogonal easy quantum group, we therefore exclude that case. Since the founding work \cite{banica2009liberation}, it is known that classifications of partition structures naturally split into four cases depending on the sizes of the blocks and we will follow this distinction hereafter.

The first case is when all the blocks of the partitions in $\CC$ have size two, such partitions being called \emph{pair partitions}. The classification then relies on a two-coloured version of the free orthogonal quantum group which we now introduce. Recall that $PO_{N}^{+}$ is the so-called \emph{projective version} (see for instance \cite[Sec 3]{banica2010invariants}) of $O_{N}^{+}$, i.e.~$C(PO_{N}^{+})$ is the subalgebra of $C(O_{N}^{+})$ generated by all the elements of the form $u_{ij}u_{kl}$ for $1\leqslant i, j, k, l\leqslant N$.

\begin{de}
The compact quantum group $O_{N}^{++}$ is defined as the amalgamated free product
\begin{equation*}
O_{N}^{++} = O_{N}^{+}\underset{PO_{N}^{+}}{\ast} O_{N}^{+}.
\end{equation*}
\end{de}

According to Proposition \ref{prop:amalgamatedstability}, $O_{N}^{++}$ is a noncrossing partition quantum group and we can give a simple generator of its category of partitions. For two (not necessarily different) colours $a, b\in \A$, we denote by $D_{ab}\in NC(ab, \emptyset)$ the partition $\sqcup$ with colours $a$ and $b$.

\begin{lem}
The category of partitions of $O_{N}^{++}$ is the category $\CC_{O^{++}}$ generated by $D_{xy}^{*}D_{xy}$.
\end{lem}

\begin{proof}
If $x$ and $y$ are the colours corresponding to the two copies of $O_{N}^{+}$, then the two copies of $PO_{N}^{+}$ that we have to identify are generated by the coefficients of $u^{xx}$ and $u^{yy}$ respectively. By the results of \cite[Ex 5.10]{freslon2013representation}, they correspond to the partitions $\pi(x, x)\otimes\pi(x, x)$ and $\pi(y, y)\otimes\pi(y, y)$. We now use Proposition Proposition \ref{prop:amalgamatedstability} with $K = \{1\}$, so that $b = \emptyset$. We therefore have to add $s = r = \pi(x, y)\otimes\pi(x, y)$. Moreover, once the generators are identified, the whole C*-algebras that they generate are identified. Thus, $O_{N}^{++}$ is the partition quantum group associated to the category of partitions $\CC_{O^{++}}$ generated by $\pi(x, y)\otimes\pi(x, y)$, which is a rotated version of $D_{xy}^{*}D_{xy}$.
\end{proof}

Note that it is known since \cite{banica2009liberation} that pair partitions are preserved under all the category operations, hence $\CC_{O^{++}}$ is a category of pair partitions. The compact quantum group $O_{N}^{++}$ is not free since $u^{x}\otimes u^{y}$ contains the non-trivial one-dimensional representation associated to the projective partition $D_{xy}^{*}D_{xy}$ which is non-trivial since $D_{xy}\notin\CC_{O^{++}}$ (this is just a rotated version of $\pi(x, y)$, which has been excluded). Let us give another description of this representation.

\begin{lem}\label{lem:cyclic1d}
Let $N\geqslant 4$ be an integer. For $1\leqslant k\leqslant N$, set
\begin{equation*}
s = \sum_{m=1}^{N}u_{km}^{x}u_{km}^{y}\in C(O_{N}^{++}).
\end{equation*}
Then, $s$ is a group-like element which does not depend on $k$. Moreover, it generates the group of group-like elements $\mathcal{G}(O_{N}^{++})$ (which is therefore cyclic).
\end{lem}

\begin{proof}
This is a standard computation. The fact that $T_{D_{xy}^{*}D_{xy}}\in \Mor(u^{x}\otimes u^{y}, u^{x}\otimes u^{y})$ is equivalent to the following relation between the generators : for any $1\leqslant i, j, k, l \leqslant N$,
\begin{equation*}
\delta_{kl}\sum_{m=1}^{N}u_{mi}^{x}u_{mj}^{y} = \delta_{ij}\sum_{m=1}^{N}u_{km}^{x}u_{lm}^{y}.
\end{equation*}
In particular, the sum does not depend on the choice of $i, j, k, l$ as soon as $i=j$ and $k=l$ and is then equal to $s$. A straightforward computation then yields $\D(s) = s\otimes s$ so that $s$ is a group-like element. Since $s$ is contained as a representation in $u^{x}\otimes u^{y}$, it is equal to $u_{D_{xy}^{*}D_{xy}}$.

Let now $t$ be a group-like element and let $p\in \CC_{O^{++}}$ be a non-through-block projective partition such that $u_{p} = t$. By concatenating any two neighbouring points with the same colour with $D_{xx}$ or $D_{yy}$, we see that $p$ is equivalent to a projective partition $p'$ in which the colours alternate in each row. Since $p'$ is alternating and of even size, it has the same colouring as $(D_{xy}^{*}D_{xy})^{\otimes \ell}$ or $(D_{yx}^{*}D_{yx})^{\otimes \ell} = \overline{(D_{xy}^{*}D_{xy})}^{\otimes \ell}$ for some $\ell\in \N$. Thus Lemma \ref{lem:equivalence1d} ensures that $t = u_{p} = u_{p'} = s^{\pm\ell}$, concluding the proof.
\end{proof}

We first prove a separate lemma to make the proof more clear. We are thankful to the referee for suggesting this simplification of the proof.

\begin{lem}\label{lem:non-through-blockorthogonal}
Let $q$ be a pair partition lying on one line. Then, for any $k\geqslant 1$ and any $a_{1}, \cdots, a_{k}, b_{1}, \cdots, b_{k}\in \{x, y\}$,
\begin{equation*}
q\otimes\pi(a_{1}, b_{1})\otimes\cdots\otimes\pi(a_{k}, b_{k})\otimes\pi(a_{k}, b_{k})\otimes\cdots\otimes\pi(a_{1}, b_{1})\otimes q^{*}\in \CC_{O^{++}}.
\end{equation*}
In particular, $q^{*}q\in \CC_{O^{++}}$.
\end{lem}

\begin{proof}
We proceed by induction on the number of blocks of $q$. If $q$ has one block, then $q = D_{ab}$ for some $a, b\in \CC_{O^{++}}$. By definition, $D_{ab}^{*}D_{ab}\in \CC_{O^{++}}$ and rotating it gives $p_{1} = D_{ab}\otimes\pi(b, a)\in \CC_{O^{++}}$ and $p_{2} = \pi(b, a)\otimes D_{ab}^{*}\in\CC_{O^{++}}$. Thus,
\begin{equation*}
D_{ab}\otimes \pi(b, a)^{\otimes (2k)}\otimes D_{ab}^{*} = p_{1}\otimes \pi(b, a)^{\otimes (2k-2)}\otimes p_{2} \in \CC_{O^{++}}
\end{equation*}
and performing vertical concatenations with $\pi(x, y)\otimes\pi(x, y)$ or $\pi(y, x)\otimes\pi(y, x)$ we get the result for $q = D_{ab}$. If now $q$ has more than one block, its endpoints are connected by assumption and form the block $D_{dc}$. If $q'$ is the complement of this block in $q$, then rotating $q\otimes\pi(a_{1}, b_{1})\otimes\cdots\otimes\pi(a_{k}, b_{k})\otimes\pi(a_{k}, b_{k})\otimes\cdots\otimes\pi(a_{1}, b_{1})\otimes q^{*}$ yields
\begin{equation*}
q'\otimes\pi(c, d)\otimes\pi(a_{1}, b_{1})\otimes\cdots\otimes\pi(a_{k}, b_{k})\otimes\pi(a_{k}, b_{k})\otimes\cdots\otimes\pi(a_{1}, b_{1})\otimes\pi(c, d)\otimes q^{\prime*}
\end{equation*}
and the result follows by induction.
\end{proof}

\begin{thm}\label{thm:biorthogonalclassification}
Let $N\geqslant 4$ be an integer and let $\CC$ be a category of noncrossing partitions which is not a free product and such that all blocks have size two. Then, there is an integer $k$ such that $C(\G_{N}(\CC))$ is the quotient of $C(O_{N}^{++})$ by group-like relations (i.e.~$s^{k}=1$).
\end{thm}

\begin{proof}
Let us first remark that since $\CC_{O_{N}^{+}}\ast\CC_{O_{N}^{+}}$ is generated by $D_{xx}$ and $D_{yy}$, this free product of categories of partitions is contained in $\CC$. We now proceed in several steps.\\

\noindent\textbf{Step 1.} We first claim that that there is a partition $p$ in $\CC$ with a block $b$ having two different colours. To show this, let us assume the converse and prove by induction on the number of blocks that any partition in $\CC$ is in fact in $\CC_{O_{N}^{+}}\ast\CC_{O_{N}^{+}}$. If $p$ has one block, the result is clear. Otherwise, by noncrossingness $p$ contains an interval $b$. We can therefore rotate $p$ to put it in the form $b\otimes q$. Since $b\in \CC$, so is $q$ and by induction it is in the free product, as well as $b\otimes q$ and its rotated version $p$. Thus, $\CC$ is a free product, contradicting the assumptions of the Theorem.\\

\noindent\textbf{Step 2.} We now prove that $\CC_{O^{++}}\subset \CC$. Indeed, there exists $p\in \CC$ such that one block of $p$ is (up to relabelling the colours) $D_{xy}$. Let us consider, among all such partitions $p$ lying on one line, one with a block $b$ of the form $D_{xy}$ such that the number of points between its endpoints is minimal. The blocks nested inside $b$ must be either $D_{xx}$ or $D_{yy}$ but these can be removed by concatenation. Thus, there is no point between the endpoints of $b$, i.e.~it is an interval. Then, by Proposition \ref{prop:intervalblockstable}, $D_{xy}^{*}D_{xy}\in \CC$ and $\CC_{O^{++}}\subset \CC$.\\

\noindent\textbf{Step 3.} Let $p\in \CC$ be a projective partition. We want to prove that it is in $\CC_{O^{++}}$. As mentioned in Remark \ref{rem:noncrossingdecomposition}, $p$ can be written as an horizontal concatenation of identity partitions and non-through-block projective partitions with endpoints connected (recall that we only have pair partitions). The result thus follows directly from Lemma \ref{lem:non-through-blockorthogonal}.\\

\noindent\textbf{Step 4.} We can now conclude : any other partition, when rotated on one line, gives a group-like relation, which is $s^{k} = 1$ for some $k$ by Lemma \ref{lem:cyclic1d}.
\end{proof}

\section{Classification II : blocks of size at most two}\label{sec:bistochastic}

The next step is to consider categories of partitions $\CC$ such that all the partitions have blocks of size one or two. Let us first recall that if there were only one colour, then according to \cite{banica2009liberation} and \cite{weber2012classification} there would be three possible compact quantum groups, namely
\begin{itemize}
\item $B_{N}^{+}$ (the quantum bistochastic group) whose category of partitions is generated by the identity partition and a singleton $\{\{1\}\}\in P(1, 0)$,
\item $B_{N}^{+\sharp} = B_{N}^{+}\ast\Z_{2}$ whose category of partitions is generated by the identity partition and the double singleton partition $\{\{1\}, \{2\}\}\in NC(2, 0)$,
\item $B_{N}^{+\prime} = B_{N}^{+}\times\Z_{2}$ whose category of partitions $\CC_{B^{+\sharp}}$ is generated by the identity partition and $\{\{1\}, \{2, 4\}, \{3\}\}\in P(4, 0)$.
\end{itemize}
It is known that $B_{N}^{+}$ is isomorphic to $O_{N-1}^{+}$ so that we could expect this step to be easily deduced from the previous one. This fails however because the non-trivial one-dimensional representations of the one-coloured case enter the picture. More precisely, there will be two families of compact quantum groups in this section, one involving only relations with the group of one-dimensional representations and the other one involving a twisted amalgamated free product.

\subsection{The non-amalgamated case}

Assume that all the partitions in $\CC$ contain only blocks of size one and two but that blocks of size two are all of the form $D_{xx}$ or $D_{yy}$. We first need some results about one-dimensional representations. For any $1\leqslant i\leqslant N$, we set
\begin{equation*}
s_{x} = \sum_{k=1}^{N}u^{x}_{ik} \text{ and } s_{y} = \sum_{k=1}^{N}u^{y}_{ik}
\end{equation*}
and for a word $w$ on $\A$ we denote by $P_{w}$ the unique partition in $NC^{\A}(w, \emptyset)$ all of whose blocks have size one.

\begin{lem}\label{lem:onedbistochastic}
Let $N\geqslant 4$ be an integer and let $\CC$ be a category of noncrossing partitions such that $P_{xx}, P_{yy}\in \CC$. Then, the elements $s_{x}$ and $s_{y}$ do not depend on $i$. Moreover, they are group-like elements satisfying $s_{x}^{2} = 1 = s_{y}^{2}$ and they generate $\mathcal{G}(\G_{N}(\CC))$.
\end{lem}

\begin{proof}
The computation is similar to that of Lemma \ref{lem:cyclic1d} and it is of course enough to do it for $x$. More precisely, rotating $P_{xx}$ we get $P_{x}^{*}P_{x}\in \CC$, implying that the sum defining $s_{x}$ does not depend on $i$. Checking that the element is group-like is straightforward. Eventually, $P_{xx}$ is an equivalence between $(P_{x}^{*}P_{x})^{\otimes 2}$ and $\emptyset$, hence $s_{x}^{2} = 1$. The fact that $s_{x}$ and $s_{y}$ generate the group of group-like elements is a direct consequence of Lemma \ref{lem:equivalence1d}.
\end{proof}

Let $p$ be a partition in $\CC$ with a block of size one which can be assumed (up to relabelling the colours) to be coloured with $x$. This block is full, hence $P_{x}^{*}P_{x}\in \CC$ by Proposition \ref{prop:intervalblockstable} and $P_{xx}\in \CC$. However, $P_{yy}$ need not be in $\CC$ and we first treat the case where $P_{yy}\notin\CC$. To do this, we introduce another compact quantum group :

\begin{de}
The C*-algebra of the compact quantum group $BO_{N}^{+\sharp}$ is defined to be the quotient of $C(B_{N}^{+\sharp}\ast O_{N}^{+})$ by the relations
\begin{equation*}
s_{x} u_{ij}^{yy} = u^{yy}_{ij}s_{x}
\end{equation*}
for all $1\leqslant i, j\leqslant \dim(u^{yy})$. %We define similarly $BO_{N}^{+\prime}$ and $BO_{N}^{+}$.
\end{de}

By Proposition \ref{prop:amalgamatedstability}, $BO_{N}^{+\sharp}$ is a noncrossing partition quantum group. Moreover, its category of partitions is generated by the free product $\CC_{B^{+\sharp}}\ast\CC_{O^{+}}$ together with the partition $P_{x}\otimes \pi(y, y)\otimes \pi(y, y)\otimes P_{x}$. The idea for the classification is that if $\CC$ is not a free product, then the corresponding compact quantum group is a quotient of $BO_{N}^{+\sharp}$.

\begin{lem}\label{lem:bistochasticnotfreeproduct}
Let $N\geqslant 4$ be an integer and let $\CC$ be a category of noncrossing partitions such that all blocks have size at most two, $P_{xx}\in \CC$, $P_{yy}\notin \CC$ and $D_{xy}$ is not a block in $\CC$. If it is not a free product, then $P_{x}\otimes \pi(y, y)\otimes \pi(y, y)\otimes P_{x}\in \CC$.
\end{lem}

\begin{proof}
Let $\CC_{x}$ be the category of all partitions in $\CC$ which are only coloured by $x$. By the classification of orthogonal easy quantum groups, it is the category of partitions of one of the three bistochastic quantum groups since by assumption $D_{xx}, P_{xx}\in \CC$. Moreover, $\CC_{y} = \CC_{O^{+}}$. Set $\DD = \CC_{x}\ast \CC_{y}$ and note that $\DD\subset\CC$. We claim that if $\CC$ is not a free product, then at least one of its projective partitions is not in $\DD$. Indeed, if all projective partitions in $\CC$ lie in $\DD$, then $\G_{N}(\CC)$ is a quotient of $\G_{N}(\DD)$ by group-like relations. But \cite[Thm 3.10]{wang1995free} implies that $\G_{N}(\DD)$ has only one non-trivial one-dimensional representation $s_{x}$, which is moreover of order two. Thus, the only possible group-like relation is $s_{x} = 1$, which is equivalent to $P_{x}\in \CC$. But then, $P_{x}\in \CC_{x}\subset\DD$ so that $\CC = \DD$.

We will now prove the statement. Let $p$ be a projective partition in $\CC$ such that any projective partition with strictly less blocks is in $\DD$. If $p$ decomposes as an horizontal concatenation $p = p'\otimes p''$, then $p', p''\in \DD$ by assumption so that $p\in \DD$, a contradiction. Thus, $p$ cannot be decomposed, which means that the endpoints of each row are connected. Rotating $p$ then yields the partition $b\otimes\pi(a, a)\otimes\pi(a, a)\otimes b^{*}$ for some partition $b\in \CC(w, \emptyset)$. By Proposition \ref{prop:intervalblockstable}, $b^{*}b\in \CC$ and since it has less blocks than $p$, $b^{*}b\in \DD$. Because the only non-trivial one-dimensional representation of $\G_{N}(\DD)$ is $s_{x} = u_{P_{x}}$, either $b\in \DD$ or $b^{*}b\sim_{\DD} P_{x}$. In the first case $p\in \DD$, a contradiction. Thus we are in the second case and using the equivalence we see that $P_{x}\otimes\pi(a, a)\otimes\pi(a, a)\otimes P_{x}\in \CC$. If $a = x$, this is a partition in $\CC_{x}$, hence in $\DD$. Thus, $a = y$ and the proof is complete.
\end{proof}

The partition $P_{x}\otimes\pi(y, y)\otimes \pi(y, y)\otimes P_{x}$ can be rotated to produce the non-through-block projective partition
\begin{center}
\begin{tikzpicture}[scale=0.5]
\draw (-1,0.5) -- (1,0.5);
\draw (-1,-0.5) -- (1,-0.5);

\draw (-1,0.5) -- (-1,1.5);
\draw (1,0.5) -- (1,1.5);
\draw (-1,-0.5) -- (-1,-1.5);
\draw (1,-0.5) -- (1,-1.5);

\draw (0,1.5) node{$\overset{x}{\circ}$};
\draw (0,-1.5) node{$\underset{x}{\circ}$};

\draw (-1,1.5) node[above]{$y$};
\draw (1,1.5) node[above]{$y$};
\draw (-1,-1.5) node[below]{$y$};
\draw (1,-1.5) node[below]{$y$};

\draw (-2,0) node[left]{$p_{\sharp} = $};
\end{tikzpicture}
\end{center}
In general, the corresponding one-dimensional representation is non-trivial and together with $s_{x}$ they generate the group of one-dimensional representations. This makes room for many possible commutation relations and group-like relations and our classification result will be that everything can be obtained in that way. Before that, let us be more precise about the one-dimensional representations of $BO_{N}^{+\sharp}$.

\begin{lem}\label{lem:1dnonamalgamatednoy}
Setting $s_{\sharp} = u_{p_{\sharp}}$, the group of one-dimensional representations of $BO_{N}^{+\sharp}$ is generated by $s_{x}$ and $s_{\sharp}$. Moreover, for any $t\in \Gr(BO_{N}^{+\sharp})$,
\begin{equation*}
tu^{yy}_{ij} = u^{yy}_{ij}t
\end{equation*}
for all $1\leqslant i, j\leqslant \dim(u^{yy})$.
\end{lem}

\begin{proof}
We first show that the group of group-like elements is generated by $s_{x}$ and $s_{\sharp}$. Let $p = q^{*}q\in\CC_{BO^{+\sharp}}$ be a non-through-block projective partition with $q$ lying on one line. Any two neighbouring points with the same colour can be removed up to equivalence by concatenation with $D_{xx}$ or $D_{yy}$. Thus, we may assume that $q$ is alternating (in the sense that the colours alternate). Moreover, since the blocks coloured by $y$ must be $D_{yy}$, there is an even number of $y$'s in $q$. Therefore, there is a horizontal concatenation of $P_{x}$ and $p_{\sharp}$ which has the same upper colouring as $q$. By Lemma \ref{lem:equivalence1d}, $p$ is equivalent to this partition, proving the first part of the statement.

As for the second part, first note that $P_{x}\otimes \pi(y, y)\otimes \pi(y, y)\otimes P_{x}^{*}$, which is a rotation of $p_{\sharp}$, yields as in the proof of Proposition \ref{prop:commutationrelations} that $s_{x}$ commutes with all the coefficients of $u^{yy}$. Moreover,
\begin{align*}
P_{x}\otimes\pi(y, y)^{\otimes 4}\otimes P_{x}^{*} & = \left[\pi(y, y)^{\otimes 2}\otimes P_{x}\otimes\pi(y, y)^{\otimes 2}\otimes P_{x}^{*}\right]\left[P_{x}\otimes\pi(y, y)^{\otimes 2}\otimes P_{x}^{*}\otimes\pi(y, y)^{\otimes 2}\right] \\
& \in \CC_{BO^{+\sharp}}
\end{align*}
and rotating this partition yields $q_{\sharp}\otimes\pi(y, y)\otimes\pi(y, y)\otimes q_{\sharp}^{*}\in \CC_{BO^{+\sharp}}$, where $q_{\sharp}$ is the upper row of $p_{\sharp}$. As before, this implies that $s_{\sharp}$ commutes with all the coefficients of $u^{yy}$. By the first part of the statement, $s_{x}$ and $s_{\sharp}$ generate the group of group-like elements, thus the commutation relation holds for all group-like elements.
\end{proof}

We are now ready for a classification statement.

\begin{prop}\label{prop:nonamalgamatednoy}
Let $N\geqslant 4$ be an integer and let $\CC$ be a category of noncrossing partitions such that all blocks have size at most two, $P_{xx}\in \CC$, $P_{yy}\notin \CC$ and $D_{xy}$ is not a block in $\CC$. If it is not a free product, then $C(\G_{N}(\CC))$ is a quotient of $C(BO_{N}^{+\sharp})$ by one or several of the following relations :
\begin{itemize}
\item $tu^{xx}_{ij} = u^{xx}_{ij}t$ for some group-like element $t$ and all $1\leqslant i, j\leqslant \dim(u^{xx})$,
\item group-like relations.
\end{itemize}
\end{prop}

\begin{proof}
\noindent\textbf{Step 1.} By Lemma \ref{lem:bistochasticnotfreeproduct}, $p_{\sharp}\in \CC$. Moreover, the same argument as in \textbf{Step 1.} of the proof of Theorem \ref{thm:biorthogonalclassification} shows that if $P_{x}\in \CC$, then all partitions with blocks of size at most two and only $x$-blocks of size one are in $\CC$. In other words, $\CC = \CC_{B^{+}}\ast\CC_{O^{+}}$ in that case. Thus, $P_{x}\notin \CC$. As before, we denote by $\CC_{x}$ the category of all partitions in $\CC$ coloured only by $x$.\\

\noindent\textbf{Step 2.} Let $\DD\subset \CC$ be the category of partitions generated by $\CC_{x}\ast\CC_{O^{+}}$ (seeing $\CC_{O^{+}}\subset \CC_{y}$) and $p_{\sharp}$. We claim that for any projective partition $p\in \CC$ there exists $\DD\subset\CC'\subset \CC$ such that $p\in \CC'$ and $\G_{N}(\CC')$ is obtained by quotienting $\G_{N}(\DD)$ by the first relation in the statement. By the same arguments as in the proof of Theorem \ref{thm:biorthogonalclassification}, it is enough to consider partitions of the form $p = q^{*}q$ where $q$ lies on one line and has its endpoints connected. Rotating yields the partition 
\begin{equation*}
r = b\otimes\pi(a, a)\otimes\pi(a, a)\otimes b^{*}\in \CC.
\end{equation*}
If $a=y$ then $r\in \DD$ by the second part of Lemma \ref{lem:1dnonamalgamatednoy}. Otherwise, if $\CC'$ is the category of partitions generated by $\DD$ and $r$, then $C(\G_{N}(\CC'))$ is the quotient of $C(\G_{N}(\DD))$ by the relations $u_{b^{*}b}u^{xx}_{ij} = u^{xx}_{ij}u_{b^{*}b}$ for all $1\leqslant i, j\leqslant \dim(u^{xx})$, hence the result.\\

\noindent\textbf{Step 3.} According to the first two steps, there exists $\DD\subset\CC'\subset\CC$ such that $C(\G_{N}(\CC'))$ is a quotient of $C(\G_{N}(\DD))$ by commutation relations and all projective partitions of $\CC$ are in $\CC'$. It follows that $C(\G_{N}(\CC))$ is a quotient of $C(\G_{N}(\CC'))$ by group-like relations.
\end{proof}

\begin{rem}
The set of possible relations in the previous statement may seem large, but it can in fact be easily described. First note that they are in fact determined by two subgroups of $\Gr(BO_{N}^{+\sharp})$ : the subgroup of elements commuting with $u^{xx}$ and the subgroup of elements made trivial, which is normal and contained in the first one. Since $\Gr(BO_{N}^{+\sharp})$ is generated by two elements of order two, it is a dihedral group and its subgroup are classified : they are generated either by $(s_{x}s_{\sharp})^{n}$ for some $n\in \N$, or by $(s_{x}s_{\sharp})^{m}s_{x}$ for some $m\in \Z$, or by the two previous elements. It is easy to see that such a subgroup is normal if and only if it is generated by the first type of elements, so that the quotients of $BO_{N}^{+\sharp}$ appearing in Proposition \ref{prop:nonamalgamatednoy} are parametrized by three integers $(d, m, n)$ with $m\mid d$ and $\vert n\vert < m$.
\end{rem}

If now both $P_{xx}$ and $P_{yy}$ belong to $\CC$, then both $s_{x}$ and $s_{y}$ may be involved in commutation relations, as well as any element of the infinite dihedral group $\Z_{2}\ast\Z_{2}$ that they generate inside $C(B_{N}^{+\sharp})\ast C(B_{N}^{+\sharp})$. Moreover, since points coloured with $y$ need not be connected any more, the proof of Lemma \ref{lem:bistochasticnotfreeproduct} breaks down so that there is no automatic commutation relation.

\begin{prop}
Let $N\geqslant 4$ be an integer and let $\CC$ be a category of noncrossing partitions with blocks of size at most two such that $P_{xx}, P_{yy}\in \CC$. If it is not a free product, then $C(\G_{N}(\CC))$ is the quotient of $C(B_{N}^{+\sharp})\ast C(B_{N}^{+\sharp})$ by one or several of the following relations
\begin{itemize}
\item $t u^{xx}_{ij} = u^{xx}_{ij}t$ for all $1\leqslant i, j \leqslant \dim(u^{xx})$, where $t$ is some group-like element,
\item $t u^{yy}_{ij} = u^{yy}_{ij}t$ for all $1\leqslant i, j \leqslant \dim(u^{yy})$, where $t$ is some group-like element,
\item group-like relations.
\end{itemize}
\end{prop}

\begin{proof}
Consider again the categories $\CC_{x}$ and $\CC_{y}$ of all partitions in $\CC$ coloured only with $x$ and $y$ respectively and set $\DD = \CC_{x}\ast\CC_{y}$. If all projective partitions of $\CC$ lie in $\DD$, then $C(\G_{N}(\CC))$ is a quotient of $C(\G_{N}(\CC_{x}))\ast C(\G_{N}(\CC_{y}))$ by group-like relations. It follows from the classification of noncrossing orthogonal easy quantum groups that each of the two factors of the free product is itself a quotient of $C(B_{N}^{+\sharp})$ by group-like relations or commutation relations, hence the result holds in that case.

Otherwise, the same argument as in the proof of Proposition \ref{prop:nonamalgamatednoy} shows that there exists $\DD\subset \CC'\subset\CC$ such that $C(\G_{N}(\CC'))$ is a quotient of $C(\G_{N}(\CC))$ by commutation relations and all projective partitions in $\CC$ lie in $\CC'$. We can then conclude as before.
\end{proof}

\begin{rem}
Once again, the statement can be made slightly more precise by using the subgroup structure of $\Gr(\G_{N}(\CC))$. The quotients will be parametrized by an integer $d$ and two pairs $(m, n)$, $(m', n')$ such that both $m$ and $m'$ divide $d$ and $\vert n\vert < m$ and $\vert n'\vert < m'$. Indeed, we have two groups of elements commuting with $u^{xx}$ and $u^{yy}$ respectively, and a common normal subgroups of elements made trivial.
\end{rem}

The last two statements are not as explicit as in the case of pair partitions. This shows that the theory of partition quantum groups is quite rich, but also raises the question of the statement of a proper classification result. We will comment more on this in Section \ref{sec:summary}.

\subsection{The amalgamated case}

We will now consider categories of noncrossing partitions $\CC$ containing at least one partition having $D_{xy}$ as a block. The idea is that this block mixes the two copies of $B_{N}^{+}$ so that the free product becomes amalgamated and the classification should be close to the case of pair partitions. The situation is however more complicated because the amalgamation can be twisted by one-dimensional representations in the sense of Definition \ref{de:twistedamalgamtion}. More precisely, consider the free product $B_{N}^{+\sharp}\ast B_{N}^{+\sharp}$. Each copy contains the subgroup $PB_{N}^{+\sharp} \simeq PO_{N-1}^{+}$ generated by the coefficients of $u^{xx}$ and $u^{yy}$ respectively and this is where twisted amalgamation takes place. %To express our result conveniently, let us say that a \emph{left class} in a group $\Gamma$ is the inverse image of a left coset with respect to some subgroup. This can be intrinsically defined as a subset $K\subset \Gamma$ such that $\{x^{-1}y \mid x, y\in K\}$ is a group. We can now give the central definition of this subsection.

\begin{de}
Let $K \subset \Z_{2}\ast\Z_{2}$ be a left coset. The C*-algebra of the compact quantum group $B_{N}^{++}(K)$ is defined to be the quotient of $C(B_{N}^{+\sharp}\ast B_{N}^{+\sharp})$ by the relation
\begin{equation*}
t u^{xx}_{ij} = u^{yy}_{ij}t
\end{equation*}
for all $1\leqslant i, j\leqslant \dim(u^{xx})$ and all $t\in K$.
\end{de}

By Proposition \ref{prop:commutationrelations}, $B_{N}^{++}(K)$ is a noncrossing partition quantum group. Let us denote by $x$ and $y$ the canonical generators of $\Z_{2}\ast\Z_{2}$. If $t\in \Z_{2}\ast\Z_{2}$ can be written as a word $g_{1}\cdots g_{n}$ on the generators, we denote by $P_{t}\in P(g_{1}\cdots g_{n}, \emptyset)$ the unique partition all of whose blocks are singletons. Then, the category of partitions $\CC_{B^{++}(K)}$ of $B_{N}^{++}(K)$ is generated by $P_{xx}$, $P_{yy}$ and $P_{t}\otimes \pi(x, y)\otimes\pi(x, y)\otimes P_{t}^{*}$ for all $t\in K$. If $K = \{1\}$ then we get the usual amalgamated free product. Using $B_{N}^{++}(K)$, we can complete the classification for blocks of size at most two.

\begin{thm}\label{thm:bibistochasticclassification}
Let $N\geqslant 4$ be an integer and let $\CC$ be a category of noncrossing partitions with blocks of size less than two and assume that $P_{xx}\in \CC$ and $D_{xy}$ is a block in $\CC$. Then, there exists a left coset $K\subset \Z_{2}\ast\Z_{2}$ such that $C(\G_{N}(\CC))$ is a quotient of $C(B_{N}^{++}(K))$ by group-like relations.
\end{thm}

\begin{proof}
\noindent\textbf{Step 1.} We first prove that $\CC_{B^{++}(K)}\subset \CC$ for some $K$. Let $p\in \CC$ be a partition containing $D_{xy}$ as a block and note that concatenating with $P_{xx}$ we can produce a $y$-singleton which is a full subpartition, hence $P_{yy}\in \CC$ by Proposition \ref{prop:intervalblockstable}. Consider now a full subpartition  $q$ of $p$ delimited by $D_{xy}$. Again by Proposition \ref{prop:intervalblockstable}, $q^{*}q\in \CC$ and rotating it gives a partition of the form $b\otimes\pi(x, y)\otimes \pi(x, y)\otimes b^{*}\in \CC$. If $b$ has colouring $g_{1},\cdots, g_{n}$ and $t = g_{1}\cdots g_{n}\in \Z_{2}\ast \Z_{2}$, then
\begin{equation*}
\left[\pi(y, y)\otimes\pi(y, y)\otimes (P_{t}^{*}P_{t}))\right](q^{*}q)\left[(P_{t}^{*}P_{t})\otimes\pi(x, x)\otimes\pi(x, x)\right]\in \CC
\end{equation*}
equals $P_{t}\otimes\pi(x, y)\otimes \pi(x, y)\otimes P_{t}^{*}$ and the claim is proved.\\

\noindent\textbf{Step 2.} Our next step is to prove that for any projective partition $p\in \CC$, there exists $K\subset K'$ such that $p\in \CC_{B^{++}(K')}\subset\CC$. To see this, we can as before focus on a non-through-block projective partition of the form $q^{*}q$ where the endpoints of $q$ are connected. If the colours of these endpoints are $x$ and $y$, then \textbf{Step 1.} of this proof shows that $q^{*}q\in \CC_{B^{++}(K')}$ for a suitable $K'$. If the colours of the endpoints are the same, say $x$, then rotating $q^{*}q$ yields the partition
\begin{equation*}
s = b\otimes \pi(x, x)\otimes \pi(x, x)\otimes \overline{b}^{*}
\end{equation*}
for some partition $b$ lying on one line. Let $t\in K$ and let $c^{*}c$ be a non-through-block projective partition such that $u_{c^{*}c} = t$. Then, $s' = \overline{c}\otimes\pi(x, y)\otimes\pi(x, y)\otimes c^{*}\in \CC_{B^{++}(K)}$ and rotating $s's$ yields
\begin{equation*}
s'' = c\otimes b\otimes \pi(x, y)\otimes\pi(x, y)\otimes c^{*}\otimes b^{*}.
\end{equation*}
Adding this partition to $\CC$ is the same as adding $u_{(c^{*}c)\otimes (b^{*}b)} = u_{c^{*}c}\otimes u_{b^{*}b}$ to $K$. A similar argument shows that reciprocally, if $s', s''\in K$ then $s\in \CC$. Summing up, we can enlarge $K$ so that $s\in \CC_{B^{++}(K')}$.\\

\noindent\textbf{Step 3.} If we consider now an arbitrary partition in $\CC$, rotating it on one line gives a relation making a one-dimensional representation trivial, i.e.~a group-like relation and the proof is complete.
\end{proof}

\section{Classification III : even partitions}\label{sec:hyperoctahedral}

When classifying orthogonal easy quantum groups, four families appear : the orthogonal one, the bistochastic one, the hyperoctahedral one and the symmetric one. The last two families are very similar since they are the two simplest examples of free wreath products. Similarly here, they will both appear through the free wreath products of pairs construction introduced in Subsection \ref{subsec:wreathpairs}. However, there is a fundamental distinction which will prove practical for our study : the presence or the absence of a block of odd size.

We will now assume that $\CC$ contains only partitions with blocks of even size. We will moreover assume that there is a partition $p$ with a block of size at least $4$. In this context, things are different from the first two cases and free wreath products are needed. For $d\in \N$, we will denote by $\Gamma_{d}$ the quotient of the group $\Z_{2}\ast\Z_{2}$ by the relation $(xy)^{d} = e$, where $x$ and $y$ are the two generators. By convention, we set $\Gamma_{0} = \Z_{2}\ast\Z_{2}$. It is easy to check that the category of partitions generated by $\pi((xy)^{d}, \emptyset)$ is exactly $\CC_{\Gamma_{d}, \{x, y\}}$. For simplicity, we will omit the generating set from now on since it is always $\{x, y\}$. To apply Proposition \ref{prop:classificationwreath} we need to know that $\CC$ already contains the category of partitions of a free wreath product of pair. Let us start with a simple lemma.

\begin{lem}\label{lem:hyperoctahedralinterval}
Let $\CC$ be a category of noncrossing partitions such that all blocks have even size. If $\CC$ contains a partition with a block of size at least four, then $\pi(xx, xx)\in \CC$ or $\pi(yy, yy)\in \CC$.
\end{lem}

\begin{proof}
Let $p\in \CC$ be a partition with a block of size at least $4$ and let $q$ be a minimal full subpartition containing all the points of this block, so that $q^{*}q\in \CC$ by Proposition \ref{prop:intervalblockstable}. We therefore have a through-block and between the points on each row, a non-through-block partition. As in the proof of \cite[Lem 4.2]{freslon2013fusion}, we can reduce $q^{*}q$ by concatenations until the through-block has only four points left while still being in $\CC$. Rotating it then yields
\begin{equation*}
r = b\otimes \pi(a_{1}a_{1}, a_{2}a_{2})\otimes \overline{b}^{*}\in \CC
\end{equation*}
for some $a_{1}, a_{2}\in \A$ and a non-through block partition $b$. Then, $\pi(a_{1}a_{1}, a_{1}a_{1}) = r^{*}r\in \CC$ and the proof is complete.
\end{proof}

With this in hand, we can already rule out the case where no partition has two colours.

\begin{lem}\label{lem:amalgamatedhyperoctahedral}
Let $\CC$ be a category of noncrossing partitions such that all blocks have even size and containing a block of size at least four. If no partition in $\CC$ has two colours, then it is a free product.
\end{lem}

\begin{proof}
From Lemma \ref{lem:hyperoctahedralinterval} we know that up to relabelling the colours, all even partitions coloured by $x$ are in $\CC$. As for $y$, we have two possibilities : either only pair partitions are in $\CC$, or all even partitions are in $\CC$. In the first case we have $\CC_{H_{N}^{+}}\ast\CC_{O_{N}^{+}}\subset\CC$ and in the second case $\CC_{H_{N}^{+}}\ast\CC_{H_{N}^{+}}\subset\CC$. Let us consider the category of all partitions such that all blocks have only one colour and all blocks have even size (resp. and all blocks coloured by $y$ have size two). In both cases, an easy induction as in \textbf{Step 1.} of the proof of Theorem \ref{thm:biorthogonalclassification} shows that this category is equal to $\CC_{H_{N}^{+}}\ast\CC_{H_{N}^{+}}$ (resp. to $\CC_{H_{N}^{+}}\ast\CC_{O_{N}^{+}})$ and since it contains $\CC$, we have equality.
\end{proof}

We are now ready to prove that the quantum groups that we are trying to classify are all quotients of the aforementioned free wreath products of pairs.

\begin{prop}\label{prop:firstquotienthyperoctahedral}
Let $\CC$ be a category of noncrossing partitions such that all blocks have even size and which is not a free product. Then, $\CC_{\Gamma_{d}}\subset \CC$ for some integer $d\in \N$.
\end{prop}

\begin{proof}
We first prove that $\pi(xy, xy)\in \CC$ by dividing into two cases. First assume that there is a partition $p\in \CC$ with an interval $b$ of size at least four with two colours. Then, $b^{*}b\in \CC$ by Proposition \ref{prop:intervalblockstable} and by the same technique as in the proof of Lemma \ref{lem:hyperoctahedralinterval} we can end up with $\pi(xy, xy)\in \CC$. Assume now that all intervals of size at least four have only one colour. Then, they can be removed by concatenation with $D_{xx}$ or $D_{yy}$ and since $\CC$ contains a block with two colours, we deduce that $D_{xy}$ is an interval. Hence, $D_{xy}^{*}D_{xy}\in \CC$ by Proposition \ref{prop:intervalblockstable}. Rotating it gives $\pi(x, y)\otimes \pi(x, y)\in \CC$ so that
\begin{equation*}
\pi(xx, yy) = \left[\pi(x, y)\otimes \pi(x, y)\right]\pi(xx, xx)\in \CC
\end{equation*}
and rotating $\pi(xx, yy)$ yields $\pi(xy, xy)\in \CC$. Setting $d = \min\{k, \beta((xy)^{k}, \emptyset)\in \CC\}$, we can now conclude that $\CC_{\Gamma_{d}}\subset \CC$.
\end{proof}

By Proposition \ref{prop:classificationwreath}, we will get free wreath products of pairs. The only thing that we do not know is what kind of subgroup $\Lambda\subset\Z_{2}\ast\Z_{2}$ can appear for the one-dimensional representations. To find it, let us set $n = \min\{k\geqslant 1, \beta((xy)^{k}, (xy)^{k})\in \CC\}$, where by convention $n=0$ if this set is empty.

\begin{lem}\label{lem:cyclic1dhyperoctahedral}
Assume that $n\geqslant 1$. For $1\leqslant i\leqslant N$, set
\begin{equation*}
s = \sum_{k=1}^{N} (u_{ik}^{x}u_{ik}^{y})^{n} \in C(\G_{N}(\CC)).
\end{equation*}
Then, $s$ is a group-like element which does not depend on $k$. Moreover,
\begin{itemize}
\item $s$ generates the group of group-like elements of $\G$,
%\item $su^{x}_{ij} = u^{x}_{ij}s^{-1}$ and $su^{y}_{ij} = u^{y}_{ij}s^{-1}$,
\item $n$ divides $d$ and $s$ has order $d/n$ (this statement being empty if $d = 0$).
\end{itemize}
\end{lem}

\begin{proof}
The fact that $\beta((xy)^{n}, (xy)^{n})\in \CC$ translates into the equation
\begin{equation*}
\sum_{k=1}^{N} (u_{ik}^{x}u_{ik}^{y})^{n} = \sum_{m=1}^{N} (u_{mj}^{x}u_{mj}^{y})^{n}
\end{equation*}
for any $1\leqslant i, j\leqslant N$ and the fact that $s$ is a group-like element are proved as in Lemma \ref{lem:cyclic1d}. Let $p=q^{*}q\in \CC$ be a projective partition with $t(p)=0$. Up to equivalence, we can assume (using $D_{xx}$ and $D_{yy}$) that its upper colouring is alternating. Let $2m$ be the number of points of $q$ and assume without loss of generality that the colouring is $(xy)^{m}$. If $m<n$, let $b$ be an interval of $q$. Then, $b^{*}b\in \CC$ by Proposition \ref{prop:intervalblockstable} and has size smaller than $n$, contradicting its minimality. Thus, $m \geqslant n$. Let $m = n\times \ell + r$ be the euclidean division of $m$ by $n$ and notice that (up to relabelling the colours)
\begin{equation*}
\left[(\beta((xy)^{n}, (xy)^{n})^{\otimes \ell}\otimes (\pi(x, x)\otimes \pi(y, y))^{\otimes r}\right]p = \beta((xy)^{n}, (xy)^{n})^{\otimes \ell}\otimes q
\end{equation*}
for some projective partition $q$ of size $2r$. If $q$ is not empty, then any of its interval again contradicts the minimality of $n$. Hence, $q = \emptyset$ and the left-hand side is an equivalence between $p$ and $\beta((xy)^{n}, (xy)^{n})^{\otimes \ell}$. This implies that the one-dimensional representation associated to $p$ is equivalent to $s^{\otimes \ell}$. By definition, $n\leqslant d$ so let $d = n\times \ell + r$ be the euclidean division of $d$ by $n$. The same reasoning as before shows that either $r = 0$ or $\beta((xy)^{r}, (xy)^{r})\in \CC$, from which the last assertion follows.

\begin{comment}
Eventually, noticing that
\begin{equation*}
\left[\beta((xy)^{n}, (xy)^{n})\otimes\pi(x, x)\right]\left[\pi(x, x)\otimes \overline{\beta}((xy)^{n}, (xy)^{n})\right]\in \CC,
\end{equation*}
is an equivalence between $\beta((xy)^{n}, (xy)^{n})\otimes \pi(x, x)$ and $\pi(x, x)\otimes\overline{\beta}((xy)^{n}, (xy)^{n})$, we see that
\begin{equation*}
su^{x}_{ij} = u^{x}_{ij}s^{-1}
\end{equation*}
for all $1\leqslant i, j\leqslant N$ (and similarly for $y$).
\end{comment}

\end{proof}

For $n$ dividing $d$, let us denote by $\Lambda_{n}$ the subgroup of $\Gamma_{d}$ generated by $(xy)^{n}$. A consequence of the previous statement is that the only subgroups of $\Gamma_{d}$ which can appear in the even case are the $\Lambda_{n}$'s. We can now complete the classification of the even case.

\begin{thm}
Let $N\geqslant 4$ be an integer and let $\CC$ be a category of noncrossing partitions which is not a free product, such that all blocks have even size and with a block of size at least four. Then, there exist two integers $d$ and $n$, with $n$ dividing $d$, such that $\G = H_{N}^{+}(\Gamma_{d}, \Lambda_{n})$.
\end{thm}

\begin{proof}
This is a direct consequence of Proposition \ref{prop:firstquotienthyperoctahedral}, Lemma \ref{lem:cyclic1dhyperoctahedral}, Lemma \ref{lem:amalgamatedhyperoctahedral} and Proposition \ref{prop:classificationwreath}.
\end{proof}

One may be surprised that there is no group-like relations in the statement. This is because in the case of free wreath products of pairs, adding group-like relations is the same as replacing $\Gamma$ and $\Lambda$ by some suitable quotient, as explained in \textbf{Step 3.} the proof of Proposition \ref{prop:classificationwreath}.

\section{Classification IV : other cases}\label{sec:symmetric}

There is one case left : when some partitions contain odd blocks of size at least three. In the orthogonal case, there are two possibilites :
\begin{itemize}
\item the quantum permutation group $S_{N}^{+}$ whose category of partitions is $NC$,
\item its symmetrized version $S_{N}^{+\prime} = S_{N}^{+}\times \Z_{2}$ whose category of partitions $NC'$ consists in all noncrossing partitions with an even number of odd blocks.
\end{itemize}
As for the bistochastic case, the main distinction will be the presence or absence of two-coloured partitions in $\CC$. If there are some, we will get free wreath products of pairs. If there is none, we will be adding commutation relations between one-dimensional representations and higher-dimensional ones, as well as group-like relations.

\subsection{The non-amalgamated case}\label{subsec:nonamalgamatedsymetric}

We start with the non-amalgamated case. We will obtain quotients of a free product by some commutation relations. As one may expect, any pair of orthogonal easy quantum groups which did not appear in Subsection \ref{sec:bistochastic} may appear here. This can be seen from the following simple lemma :

\begin{lem}\label{lem:symetrichasfourblock}
Let $\CC$ be a category of noncrossing partitions such that all blocks have only one colour, with an odd block of size at least three. Then, up to exchanging the colours $P_{xx}\in \CC$ and $\CC$ contains a four-block.
\end{lem}

\begin{proof}
Recall that we are assuming that all blocks have only one colour. Let $p\in \CC$, let $b$ be an odd block of $p$ of size at least three and assume without loss of generality that it is coloured by $x$. If it is an interval, then $b^{*}b\in \CC$ and concatenating with $D_{xx}$ yields $P_{xx}\in \CC$. Assume that no odd block of size at least three is an interval. Since any even block can be removed by concatenating with $D_{xx}$ or $D_{yy}$, this means that there are two points of $b$ between which there is a singleton. Thus, we may assume again that $P_{xx}\in \CC$.

Let now $p$ be a partition containing a block of size at least three and let $q$ be a minimal the full subpartition containing this block. By Proposition \ref{prop:intervalblockstable}, $q^{*}q\in \CC$ and the same argument as in Lemma \ref{lem:hyperoctahedralinterval} yields a four-block in $\CC$.
\end{proof}

We will now split the classification into several cases, depending on the category $\CC_{y}$ of all partitions in $\CC$ coloured only with $y$. 

\subsubsection{First case}

We start with the case where blocks in $\CC_{y}$ have size less than two. We will show that the only possible commutation relations are those involving $u^{xx}$.

\begin{prop}\label{prop:symmetricbybistochastic}
Let $N\geqslant 4$ be an integer and let $\CC$ be a category of noncrossing partitions such that $P_{xx}, \pi(xx, xx)\in \CC$ and all blocks coloured only with $y$ have size at most two. If $\CC$ is not a free product, then there exist two orthogonal easy quantum groups $\G_{1}\in \{S_{N}^{+}, S_{N}^{+\prime}\}$ and $\G_{2}\in \{O_{N}^{+}, B_{N}^{+\sharp}, B_{N}^{+\prime}, B_{N}^{+}\}$ such that $C(\G_{N}(\CC))$ is a quotient of $C(\G_{1})\ast C(\G_{2})$ by one or several of the following relations
\begin{itemize}
\item $t u^{xx}_{ij} = u^{xx}_{ij}t$ for all $1\leqslant i, j \leqslant \dim(u^{xx})$, where $t$ is some group-like element,
\item $t u^{yy}_{ij} = u^{yy}_{ij}t$ for all $1\leqslant i, j \leqslant \dim(u^{yy})$, where $t$ is some group-like element,
\item group-like relations.
\end{itemize}
\end{prop}

\begin{proof}
Set $\DD = \CC_{x}\ast\CC_{y}$. We will first prove that there is a category of partitions $\CC'\subset\CC$ such that $C(\G_{N}(\CC'))$ is a quotient of $C(\G_{N}(\DD))$ by commutation relations as in the statement and any projective partition in $\CC$ lies in $\CC'$.\\

\noindent\textbf{Step 1.} We start with a non-through-block projective partition $p = q^{*}q$ where the endpoints of $q$ are connected. If the endpoints of $q$ form a two-block, then we conclude as in Proposition \ref{prop:nonamalgamatednoy} for the bistochastic case. Otherwise, the endpoints must be coloured by $x$ so let us assume that we have an $x$-block of size $n$ with other partitions $b_{1}, \cdots, b_{n-1}$ between its points. As in \cite[Lem 4.2]{freslon2013fusion}, after rotating and capping we get $r_{i} = b_{i}\otimes \pi(xx, xx)\otimes \overline{b}_{i}^{*}\in \CC$ for any $1\leqslant i\leqslant n-1$. Moreover, by the assumptions and the classification of orthogonal easy quantum groups, $\pi(xx,x)\otimes P_{x}\in \CC_{x}\subset\CC$ so that
\begin{equation*}
s_{i} = (\pi(xx, x)\otimes P_{x}^{*})r_{i}(b_{i}^{*}b_{i}\otimes\pi(x, xx)\otimes P_{x}\otimes \overline{b}_{i}\overline{b}_{i}^{*}) = b_{i}\otimes \pi(x, x)\otimes P_{x}\otimes \overline{b}_{i}^{*}\otimes P_{x}^{*}\in \CC
\end{equation*}
Now, we can simplify $s_{i}\otimes \overline{s}_{i}^{*}$ by concatenating neighbouring points with the same colour until we get the partition $t_{i} = b_{i}\otimes \pi(x, x)\otimes \pi(x, x)\otimes \overline{b}_{i}^{*}$. Thus adding $p$ implies adding $t_{i}$ for all $1\leqslant i\leqslant n$, and each $t_{i}$ yields a commutation relation as in the statement.

We now claim that the category of partitions $\CC'$ generated by $\DD$ and all the $t_{i}$'s also contains $p$. This can be proved by induction on the number $n$ of points in the large $x$-block of $q$. For $n=2$, $p$ is a rotated version of $t_{1}$. For $n\geqslant 3$, let $q'$ be the partition with an $x$-block of $n-1$ points and the partitions $b_{1}, \cdots, b_{n-2}$ in between, i.e.~$q$ with the $b_{n-1}$-part removed. By the induction hypothesis, $q^{\prime\ast}q'\in \CC'$. Recalling that $\pi(xx, x)\otimes\pi(xx, x)\in NC'$, let us set
\begin{equation*}
z = \left[\pi(xx, x)\otimes \pi(xx, x)\right]\left[\pi(x, x)\otimes t_{n-1}\otimes \pi(x, x)\right] \in \CC'.
\end{equation*}
Rotating $q^{\prime\ast}q'$ on one line and concatenating in the middle by $z$ then produces a rotated version of $p$.\\

\noindent\textbf{Step 2.} We now consider a general projective partition and write it as an horizontal concatenation of non-through-block partitions with endpoints connected and through-block partitions with endpoints connected as in Remark \ref{rem:noncrossingdecomposition}. We have to focus on the second type. If the endpoints are coloured by $y$, this is just $\pi(y, y)$ so that we assume that the endpoints are coloured by $x$. We therefore have $\pi(x^{n}, x^{n})$ with partitions $b_{1}, \cdots, b_{n-1}$ between the points. Using again \cite[Lem 4.2]{freslon2013fusion}, we see that $r_{i} = b_{i}\otimes\pi(xx, xx)\otimes\overline{b}_{i}^{*}\in \CC$ for all $1\leqslant i\leqslant n$. Let $\CC'$ be the category of partitions generated by $\DD$ and all the $r_{i}$'s. By the same proof as in \textbf{Step 1.} above, $q^{*}q\in \CC$ where $q$ lie on one line and consists in an $x$-block of size $n$ with the partitions $b_{1}, \cdots, b_{n}$ in between the points. Now, concatenating $q\otimes\overline{q}^{*}$ with $\pi(xx, xx)$ in the middle yields a rotated version of our initial projective partition, hence the result.\\

\noindent\textbf{Step 3.} The proof concludes as usual : all we can do now is add group-like relations.
\end{proof}

\subsubsection{Second case}

In this second case, we will assume that $P_{xx}\in \CC$ and that there is a partition with a $y$-block of even size at least four. The same argument as in Lemma \ref{lem:symetrichasfourblock} then shows that $\pi(yy, yy)\in\CC$ so that any $y$-coloured partition with all blocks of even size is in $\CC$. It turns out that this is the most complicated case of the classification. To explain this, we need some notations. Let $p_{k}$ be the partition $\pi(y^{k+1}, y^{k+1})$ together with an $x$-singleton between each point of the upper and lower row (but not between the endpoints of each row) and let $q_{k}$ be the non-through-block version of $p_{k}$. Here is a pictorial description :

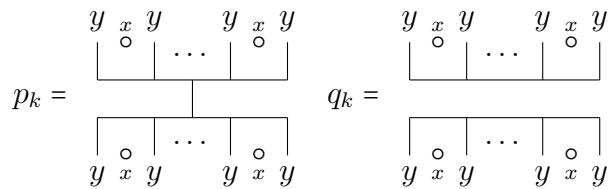
\begin{center}
\begin{tikzpicture}[scale=0.5]
\draw (-2.5,0.5) -- (2.5,0.5);
\draw (-2.5,-0.5) -- (2.5,-0.5);

\draw (-2.5,0.5) -- (-2.5,1.5);
\draw (2.5,0.5) -- (2.5,1.5);
\draw (-2.5,-0.5) -- (-2.5,-1.5);
\draw (2.5,-0.5) -- (2.5,-1.5);

\draw (-1,0.5) -- (-1,1.5);
\draw (1,0.5) -- (1,1.5);
\draw (-1,-0.5) -- (-1,-1.5);
\draw (1,-0.5) -- (1,-1.5);

\draw (0,0.5) -- (0,-0.5);

\draw (0,1.5) node[below]{$\dots$};
\draw (0,-1.5) node[above]{$\dots$};

\draw (-2.5,1.5) node[above]{$y$};
\draw (-1.75,1) node[above]{$\overset{x}{\circ}$};
\draw (-1,1.5) node[above]{$y$};
\draw (1,1.5) node[above]{$y$};
\draw (1.75,1) node[above]{$\overset{x}{\circ}$};
\draw (2.5,1.5) node[above]{$y$};

\draw (-2.5,-1.5) node[below]{$y$};
\draw (-1.75,-2.5) node[above]{$\underset{x}{\circ}$};
\draw (-1,-1.5) node[below]{$y$};
\draw (1,-1.5) node[below]{$y$};
\draw (1.75,-2.5) node[above]{$\underset{x}{\circ}$};
\draw (2.5,-1.5) node[below]{$y$};

\draw (-3,0) node[left]{$p_{k} = $};
\end{tikzpicture}
\begin{tikzpicture}[scale=0.5]
\draw (-2.5,0.5) -- (2.5,0.5);
\draw (-2.5,-0.5) -- (2.5,-0.5);

\draw (-2.5,0.5) -- (-2.5,1.5);
\draw (2.5,0.5) -- (2.5,1.5);
\draw (-2.5,-0.5) -- (-2.5,-1.5);
\draw (2.5,-0.5) -- (2.5,-1.5);

\draw (-1,0.5) -- (-1,1.5);
\draw (1,0.5) -- (1,1.5);
\draw (-1,-0.5) -- (-1,-1.5);
\draw (1,-0.5) -- (1,-1.5);

%\draw (0,0.5) -- (0,-0.5);

\draw (0,1.5) node[below]{$\dots$};
\draw (0,-1.5) node[above]{$\dots$};

\draw (-2.5,1.5) node[above]{$y$};
\draw (-1.75,1) node[above]{$\overset{x}{\circ}$};
\draw (-1,1.5) node[above]{$y$};
\draw (1,1.5) node[above]{$y$};
\draw (1.75,1) node[above]{$\overset{x}{\circ}$};
\draw (2.5,1.5) node[above]{$y$};

\draw (-2.5,-1.5) node[below]{$y$};
\draw (-1.75,-2.5) node[above]{$\underset{x}{\circ}$};
\draw (-1,-1.5) node[below]{$y$};
\draw (1,-1.5) node[below]{$y$};
\draw (1.75,-2.5) node[above]{$\underset{x}{\circ}$};
\draw (2.5,-1.5) node[below]{$y$};

\draw (-3,0) node[left]{$q_{k} = $};
\end{tikzpicture}
\end{center}

These partitions will be our key tool for the classification. Here are some elementary properties :

\begin{lem}\label{lem:pk}
Let $\CC$ be a category of noncrossing partitions such that $P_{xx}\in \CC$ and all partitions coloured by $y$ have even size. The following are equivalent :
\begin{enumerate}
\item $p_{k}\in \CC$ for all $k\in \N$,
\item $p_{k}\in \CC$ for some $k\in \N$,
\item $P_{x}\otimes \pi(yy, yy)\otimes P_{x}^{*}\in \CC$.
\end{enumerate}
\end{lem}

\begin{proof}
$(1)\Rightarrow (2)$ is clear, so let us assume that $p_{k}\in\CC$ for some $k$. Rotating it on one line and concatenating as in the proof of \cite[Lem 4.1]{freslon2013fusion}, we get $p_{1}\in \CC$, which is a rotated version of $P_{x}\otimes\pi(yy, yy)\otimes P_{x}^{*}$.

If now we assume $(3)$, we already noted that rotating we get $p_{1}\in \CC$. Moreover, using this partition we can make singletons "jump" over two $y$-points which are connected. More precisely, if a partition is in $\CC$, the same partition where an $x$-singleton has been moved over two consecutive connected $y$-points is also in $\CC$. Using this, we will prove by induction that $p_{k}\in\CC$. For $k=1$ there is nothing to prove. For $k\geqslant 1$, let us rotate $p_{k}$ on one line and denote the corresponding partition by $p'_{k}$. Then, concatenating the last point of $p'_{k}$ by $\pi(yyy, y)$ yields a partition $s$ with a $y$-block of size $2k+2$ and $x$-singletons between the first $k$ points and the $k+2$-th to $2k$-th points. This is the same as a rotation of $p_{k+1}$ except that between the last two $y$-points there is nothing. Now consider $s\otimes P_{xx}\in \CC$. Then, using the "jumping property" we obtain a rotated version of $p_{k+1}$. Here is an illustration for the process building $p_{2}$ from $p_{1}$ :
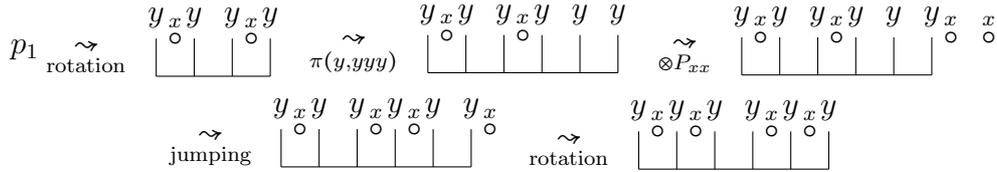
\begin{center}
\begin{tikzpicture}[scale=0.5]
\draw (-1.5,-0.5) -- (1.5,-0.5);

\draw (-1.5,-0.5) -- (-1.5,0.5);
\draw (-0.5,-0.5) -- (-0.5,0.5);
\draw (0.5,-0.5) -- (0.5,0.5);
\draw (1.5,-0.5) -- (1.5,0.5);

\draw (-1.5,0.5) node[above]{$y$};
\draw (-1,0) node[above]{$\overset{x}{\circ}$};
\draw (-0.5,0.5) node[above]{$y$};
\draw (0.5,0.5) node[above]{$y$};
\draw (1,0) node[above]{$\overset{x}{\circ}$};
\draw (1.5,0.5) node[above]{$y$};

\draw (-2,0) node[left]{$p_{1}\underset{\text{rotation}}{\leadsto}$};
\end{tikzpicture}
\begin{tikzpicture}[scale=0.5]
\draw (-1.5,-0.5) -- (3.5,-0.5);

\draw (-1.5,-0.5) -- (-1.5,0.5);
\draw (-0.5,-0.5) -- (-0.5,0.5);
\draw (0.5,-0.5) -- (0.5,0.5);
\draw (1.5,-0.5) -- (1.5,0.5);
\draw (2.5,-0.5) -- (2.5,0.5);
\draw (3.5,-0.5) -- (3.5,0.5);

\draw (-1.5,0.5) node[above]{$y$};
\draw (-1,0) node[above]{$\overset{x}{\circ}$};
\draw (-0.5,0.5) node[above]{$y$};
\draw (0.5,0.5) node[above]{$y$};
\draw (1,0) node[above]{$\overset{x}{\circ}$};
\draw (1.5,0.5) node[above]{$y$};
\draw (2.5,0.5) node[above]{$y$};
\draw (3.5,0.5) node[above]{$y$};

\draw (-2,0) node[left]{$\underset{\pi(y, yyy)}{\leadsto}$};
\end{tikzpicture}
\begin{tikzpicture}[scale=0.5]
\draw (-1.5,-0.5) -- (3.5,-0.5);

\draw (-1.5,-0.5) -- (-1.5,0.5);
\draw (-0.5,-0.5) -- (-0.5,0.5);
\draw (0.5,-0.5) -- (0.5,0.5);
\draw (1.5,-0.5) -- (1.5,0.5);
\draw (2.5,-0.5) -- (2.5,0.5);
\draw (3.5,-0.5) -- (3.5,0.5);

\draw (-1.5,0.5) node[above]{$y$};
\draw (-1,0) node[above]{$\overset{x}{\circ}$};
\draw (-0.5,0.5) node[above]{$y$};
\draw (0.5,0.5) node[above]{$y$};
\draw (1,0) node[above]{$\overset{x}{\circ}$};
\draw (1.5,0.5) node[above]{$y$};
\draw (2.5,0.5) node[above]{$y$};
\draw (3.5,0.5) node[above]{$y$};
\draw (4,0) node[above]{$\overset{x}{\circ}$};
\draw (5,0) node[above]{$\overset{x}{\circ}$};

\draw (-2,0) node[left]{$\underset{\otimes P_{xx}}{\leadsto}$};
\end{tikzpicture}
\begin{tikzpicture}[scale=0.5]
\draw (-1.5,-0.5) -- (3.5,-0.5);

\draw (-1.5,-0.5) -- (-1.5,0.5);
\draw (-0.5,-0.5) -- (-0.5,0.5);
\draw (0.5,-0.5) -- (0.5,0.5);
\draw (1.5,-0.5) -- (1.5,0.5);
\draw (2.5,-0.5) -- (2.5,0.5);
\draw (3.5,-0.5) -- (3.5,0.5);

\draw (-1.5,0.5) node[above]{$y$};
\draw (-1,0) node[above]{$\overset{x}{\circ}$};
\draw (-0.5,0.5) node[above]{$y$};
\draw (0.5,0.5) node[above]{$y$};
\draw (1,0) node[above]{$\overset{x}{\circ}$};
\draw (1.5,0.5) node[above]{$y$};
\draw (2,0) node[above]{$\overset{x}{\circ}$};
\draw (2.5,0.5) node[above]{$y$};
\draw (3.5,0.5) node[above]{$y$};
\draw (4,0) node[above]{$\overset{x}{\circ}$};

\draw (-2,0) node[left]{$\underset{\text{jumping}}{\leadsto}$};
\end{tikzpicture}
\begin{tikzpicture}[scale=0.5]
\draw (-1.5,-0.5) -- (3.5,-0.5);

\draw (-1.5,-0.5) -- (-1.5,0.5);
\draw (-0.5,-0.5) -- (-0.5,0.5);
\draw (0.5,-0.5) -- (0.5,0.5);
\draw (1.5,-0.5) -- (1.5,0.5);
\draw (2.5,-0.5) -- (2.5,0.5);
\draw (3.5,-0.5) -- (3.5,0.5);

\draw (-1.5,0.5) node[above]{$y$};
\draw (-1,0) node[above]{$\overset{x}{\circ}$};
\draw (-0.5,0.5) node[above]{$y$};
\draw (0,0) node[above]{$\overset{x}{\circ}$};
\draw (0.5,0.5) node[above]{$y$};
\draw (1.5,0.5) node[above]{$y$};
\draw (2,0) node[above]{$\overset{x}{\circ}$};
\draw (2.5,0.5) node[above]{$y$};
\draw (3,0) node[above]{$\overset{x}{\circ}$};
\draw (3.5,0.5) node[above]{$y$};

\draw (-2,0) node[left]{$\underset{\text{rotation}}{\leadsto}$};
\end{tikzpicture}
\end{center}
\end{proof}

The idea for the classification is as follows : if the partition quantum group is not a free product, then its category of partition contains $P_{x}\otimes \pi(yy, yy)\otimes P_{x}^{*}$. To get all projective partitions, one may have to add $q_{k}$ for some $k$, and then group-like relations. However, the proof of this involves the structure of $x$-coloured partitions, which is not easy to deal with. To keep things clear, we will therefore split the arguments into several statements but before we prove the claim at the beginning of this paragraph. From now on, we write $NC_{ev}$ for the category of all noncrossing partitions with blocks of even size.

\begin{lem}\label{lem:nonamalgamatedsymetricnotfree}
Let $\CC$ be a category of noncrossing partitions with $P_{x}\in \CC$ and such that $\CC_{y} = NC_{ev}$. If it is not a free product, then $P_{x}\otimes \pi(yy, yy)\otimes P_{x}\in \CC$.
\end{lem}

\begin{proof}
Writing $\DD = \CC_{x}\ast NC_{ev}$, the same argument as in Lemma \ref{lem:bistochasticnotfreeproduct} shows that if $\CC$ is not a free product then there is a projective partition in $\CC\setminus\DD$. Let $p$ be such a projective partition of minimal size. It cannot be decomposed as a tensor product $p = p'\otimes p''$ by minimality, so that the endpoints of each row must be connected. We therefore have a block with points coloured $a$ and between these points partitions $b_{i}$ lying on one line. By minimality $b_{i}^{*}b_{i}\in \DD$ so that again as in Lemma \ref{lem:bistochasticnotfreeproduct} we can either remove $b_{i}$ or replace it with $P_{x}$. Since $p\notin\DD$, there is at least one $P_{x}$ and $a=y$.

Let us assume that the $y$-block is a through-block. Then, using \cite[Lem 4.1]{freslon2013fusion} we can rotate and concatenate until we have a rotated  version of $P_{x}\otimes \pi(yy, yy)\otimes P_{x}$. If now the $y$-block is not a through-block, there are again two cases. If the $y$-blocks have size two, we have a rotated version of $P_{x}\otimes\pi(y, y)\otimes\pi(y, y)\otimes P_{x}$ and concatenating with $\pi(yy, yy)$ yields the result. Otherwise, since there is at least one $i$ such that $b_{i} = P_{x}$, by \cite[Lem 4.2]{freslon2013fusion} we can produce a rotated version of $P_{x}\otimes\pi(yy, yy)\otimes P_{x}$ in $\CC$.
\end{proof}

For the classification, let us first assume that $q_{k}\notin\CC$ for any $k\in\N$. Then, there is nothing that one can add. To state this more precisely, let us set $v=u_{\pi(yy, yy)}$.

\begin{prop}
Let $N\geqslant 4$ be an integer and let $\CC$ be a category of noncrossing partitions containing an odd block coloured with $x$ and such that $\CC_{y} = NC_{ev}$. If $q_{k}\notin\CC$ for all $k\in \N$ and $\CC$ is not a free product, then there exists $\G\in\{S_{N}^{+}, S_{N}^{+\prime}, B_{N}^{+}, B_{N}^{+\prime}, B_{N}^{+\sharp}\}$ such that $C(\G_{N}(\CC))$ is the quotient of $C(\G)\ast C(H_{N}^{+})$ by the relations
\begin{equation*}
s_{x}v_{ij} = v_{ij}s_{x}
\end{equation*}
for all $1\leqslant i, j\leqslant \dim(v)$.
\end{prop}

\begin{proof}
\noindent\textbf{Step 1.} Let $\DD$ be the category of partitions generated by $\CC_{x}\ast NC_{ev}$ and $P_{x}\otimes\pi(yy, yy)\otimes P_{x}^{*}$. We claim that all projective partitions in $\CC$ lie in $\DD$. By contradiction, consider a non-through-block projective partition $p = q^{*}q$ of minimal size among those not belonging to $\DD$. As in the proof of Lemma \ref{lem:nonamalgamatedsymetricnotfree}, we can assume that $q$ is a $y$-block with either nothing or $P_{x}$ between the points. Let us prove by induction on the number of points that such a partition is always in $\DD$. If there are two $y$-points then the partition  cannot contain a singleton since it would be equal to $q_{1}\notin\CC$. Thus there is nothing between the points and it therefore belongs to $\DD$. If there are more points, there cannot be singletons between each point since otherwise it would be equal to $q_{k}\notin\CC$. Thus there are at least two neighbouring $y$-points. We can then remove them by concatenating with $D_{yy}$, producing an equivalent partition which, by assumption, is in $\DD$. The original partition can be recovered from this one by concatenating with $\pi(y, yyy)$, hence is also in $\DD$.\\

\noindent\textbf{Step 2.} Consider now a through-block projective partition $p$ with $t(p) = 1$. Again we can assume that the through-block is coloured with $y$ and that there are either nothing or $x$-singletons between the points and the same argument as in Lemma \ref{lem:pk} shows that $p\in \DD$. Since any projective partition in $\CC$ can be written as an horizontal concatenation of those studied in \textbf{Step 1.} and \textbf{Step 2.}, we conclude that they all belong to $\DD$.\\

\noindent\textbf{Step 3.} As a consequence of the two previous steps, $C(\G_{N}(\CC))$ is a quotient of $C(\G)\ast C(H_{N}^{+})$ by the relations in the statement plus some group-like relations. However, we have seen in \textbf{Step 1.} that non-through-block projective partitions are always equivalent either to a $y$-block (which yields the trivial representation) or to an $x$-singleton. All we can do is therefore add $P_{x}$ to $\CC$, but this is already in $\CC_{x}\subset\CC$.
\end{proof}

We now assume that $q_{k}\in \CC$ for some $k \geqslant 1$. Let us first see which other $q_{k'}$'s can then be in $\CC$.

\begin{lem}\label{lem:qpartitions}
Let $\CC$ be a category of noncrossing partitions such that $q_{k}\in \CC$ for some $k$. Then, the following hold :
\begin{enumerate}
\item $p_{1}\in \CC$.
\item If $q_{k}, q_{k'}\in \CC$, then $q_{k+k'+1}\in \CC$ and $q_{k+k'+1}\sim q_{k}\otimes P_{x}^{*}P_{x}\otimes q_{k'}$,
\item If $q_{k}, q_{k'}\in \CC$ for $k>k'$, then $q_{k-k'-1}\in \CC$ and $q_{k}\sim q_{k'}\otimes P_{x}^{*}P_{x}\otimes q_{k-k'-1}$.
\end{enumerate}
As a consequence, there exists an integer $k_{0}$ such that $\{k \mid q_{k}\in\CC\} = \{nk_{0} + n-1 \mid n\in \N\}$.
\end{lem}

\begin{proof}
To prove $(1)$ we distinguish two cases. If $q_{1}\in \CC$, then rotating it we get $P_{x}\otimes\pi(y, y)\otimes\pi(y, y)\otimes P_{x}$. Concatenating with $\pi(yy, yy)$ then yields a rotated version of $p_{1}$. If $q_{k}\in \CC$ for some $k\geqslant 2$, rotating it on one line and concatenating with $D_{yy}$ we can glue the blocks to produce a rotated version of $p_{k-1}$ and the result again holds by Lemma \ref{lem:pk}.

For $(2)$, note that
\begin{align*}
r & = \left[q_{k}\otimes P_{x}^{*}P_{x}\otimes q_{k'}\right]\left[\pi(y, y)\otimes\pi(x, x))^{\otimes k}\otimes p_{1}\otimes (\pi(y, y)\otimes\pi(x, x))^{\otimes k'}\otimes\pi(y, y)\right] \\
& \in \CC
\end{align*}
satisfies $r^{*}r = q_{k+k'+1}$ and $rr^{*} = q_{k}\otimes P_{x}^{*}P_{x}\otimes q_{k'}$, hence the result. For $(3)$, simply concatenate $q_{k}$ with $q_{k'}\otimes (\pi(x, x)\otimes\pi(y, y))^{\otimes 2(k-1)}$ to get the equivalence.

Now, if $k_{0}$ is the smallest $k$ such that $q_{k}\in \CC$, it follows from point $(3)$ that any other $k$ is of the form $nk_{0} + (n-1)$ while point $(2)$ shows that all such $k$'s occur.
\end{proof}

The classification result of this subsection is that once $q_{k}$ has been added to the free product, all we can do is add commutation relations with $u^{xx}$ and group-like relations. To prove this, we will have to show that some projective partitions in a category belong to a given subcategory. For the sake of clarity, we split this argument using the following definition.

\begin{de}
Let $NC_{1, 2}$ be the category of all noncrossing partitions with blocks of size at most two. A projective partition $p$ is said to be \emph{$x$-simple} if it is equivalent in $NC_{1, 2}\ast NC_{ev}$ to a partition in which all $x$-blocks are singletons.
\end{de}

\begin{lem}\label{lem:xsimplepartitions}
Let $\CC$ be a category of noncrossing partitions containing $q_{k}$ for some $k > 0$ and let $\CC'\subset \CC$ be the category of partitions generated by $NC_{1, 2}\ast NC_{ev}$ and $q_{k}$. Then, any $x$-simple projective partition in $\CC$ lies in $\CC'$.
\end{lem}

\begin{proof}
By contradiction, let $p$ be a projective $x$-simple partition in $\CC\setminus\CC'$ of minimal size. By minimality it cannot be written as a tensor product of two projective partitions so that its endpoints are connected. Because $p$ is $x$-simple, we may assume that all $x$-points are singletons, hence the endpoints are coloured by $y$. Let us first assume that $t(p) = 0$ and write $p = q^{*}q$. Thus, $q$ consists in a $y$-block with a partition $b_{i}$ between the points. If $b_{i} = P_{x}$ for all $i$, then we have $q_{k}$ which is already in $\CC'$. Otherwise, we can concatenate with $D_{yy}$ to produce a projective partition which is equivalent to $p$ in $NC_{1, 2}\ast NC_{ev}$ but smaller, a contradiction. Assuming now that $t(p) = 1$, the same argument works : if $b_{i} = P_{x}$ for all $i$ then $p = p_{k}$ which is in $\CC'$ by Lemma \ref{lem:qpartitions} and otherwise we can reduce $p$, contradicting minimality.
\end{proof}

We will now complete the classification. To do this, we have to consider general projective partitions in $\CC$, which should amount to commutation relations and then add group-like relations. The case where $x$-blocks have size at most two is simpler, so that we treat it separately. Before that, note that adding $q_{k}$ may be seen as a commutation relation. Indeed, thanks to Lemma \ref{lem:qpartitions} we may always assume that we are adding $q_{k}$ for some odd $k = 2k'-1$. Then, rotating it gives an equivalence between $P_{x}\otimes p_{k'-1}\otimes p_{k'-1}$ and $p_{k'-1}\otimes p_{k'-1}\otimes P_{x}$. Let us introduce some notations to make things clear :

\begin{de}
Let $\CC_{x}\subset NC_{1, 2}$ and let $\CC$ be a category of noncrossing partitions containing $\CC_{x}\ast NC_{ev}$ and $q_{2k-1}$ for some $k > 0$. We will denote by $w^{(k)}$ the following representation of $\G_{N}(\CC)$ :
\begin{equation*}
w^{(k)} = u_{p_{k-1}\otimes p_{k-1}}
\end{equation*}
with the convention that $w^{(1)} = u^{y}\otimes u^{y}$.
\end{de}

We can now state our result for $\CC_{x}\subset NC_{1, 2}'$, where $NC_{1, 2}'$ is the set of partitions with blocks of size at most two with an even number of odd blocks.

\begin{prop}\label{prop:bistochasticbyhyperoctahedral}
Let $N\geqslant 4$ and let $\CC$ be a category of noncrossing partitions such that all $x$-blocks have size at most two, $\CC_{y} = NC_{ev}$ and $q_{k}\in\CC$. Then, there exists $\G\in \{B_{N}^{+\sharp}, B_{N}^{+\prime}, B_{N}^{+}\}$ such that $C(\G_{N}(\CC))$ is a quotient of $C(\G)\ast C(H_{N}^{+})$ by the relations
\begin{equation*}
s_{x}w^{(k)}_{ij} = w^{(k)}_{ij}s_{x}
\end{equation*}
for all $1\leqslant i, j\leqslant \dim(w^{(k)})$ and one or several of the following relations :
\begin{itemize}
\item $t u^{xx}_{ij} = u^{xx}_{ij}t$ for all $1\leqslant i, j\leqslant \dim(u^{xx})$, where $t$ is some group-like element,
\item group-like relations.
\end{itemize}
\end{prop}

\begin{proof}
Let $\DD$ be the category of partitions generated by $\CC_{x}\ast NC_{ev}$ and $q_{k}$. We claim that there exists a category of partitions $\DD\subset \CC'\subset \CC$ such that $\G_{N}(\CC')$ is a quotient of $\G_{N}(\DD)$ by commutation relations with $u^{xx}$ and all projective partitions in $\CC$ lie in $\CC'$.

We will prove this by induction on the number of blocks of $p$. If $p$ has one block, it is in $\DD$. Assume now that $p$ has $n\geqslant 2$ blocks. If it can be written as a horizontal concatenation, then we conclude by induction. We may therefore consider that the endpoints of $p$ on each row are connected. If they are coloured by $x$, this means that we have two $x$-points connected and another partition $b$ in the middle forming a partition $q$ and $p = q^{*}q$. But rotating yields $b\otimes\pi(x, x)\otimes \pi(x, x)\otimes b^{*}$ and we just have to add this to $\CC'$. If they are coloured by $y$, let us distinguish two cases :
\begin{itemize}
\item Assume first that $t(p) = 0$ and $p = q^{*}q$ and note that we can assume that the colouring of $q$ is alternating. If $q$ has $n = 2k'+1$ points, then $p_{k'}pp_{k'} = q_{k'}$ so that $q_{k'}\in \DD$. Moreover, if we denote by $b_{1}, \cdots, b_{n-1}$ the partitions between the points of $q$, then $b_{i}^{*}b_{i}\in\CC'$ by induction and
\begin{equation*}
r = q_{k'}\left[\pi(y, y)\otimes (b_{1}^{*}b_{1})\otimes\pi(y, y)\otimes\cdots\otimes(b_{n-1}^{*}b_{n-1})\otimes\pi(y, y)\right]\in \CC'.
\end{equation*}
Since $r^{*}r = p$ the result is proved.
\item Assume now that $t(p) = 1$. The same argument as above works : all the partitions $b_{i}^{*}b_{i}$ are in $\CC'$ by induction and using $p_{k'}$ (where $n=2k'+1$) instead of $q_{k'}$ in the definition of $r$ yields the result.
\end{itemize}
Now, all we can do is add group-like relations, hence the result.
\end{proof}

Still assuming $\CC_{y} = NC_{ev}$, the case left is when $\CC_{x}$ contains $P_{xx}$ and partitions of size at least three.

\begin{prop}
Let $N\geqslant 4$ be an integer and let $\CC$ be a category of noncrossing partitions containing $P_{xx}$, a partition with an $x$-block of size at least three and such that $\CC_{y} = NC_{ev}$ and $q_{k}\in\CC$. Then, there exists $\G\in \{S_{N}^{+\prime}, S_{N}^{+}\}$ such that $C(\G_{N}(\CC))$ is a quotient of $C(\G)\ast C(H_{N}^{+})$ by one or several of the following relations :
\begin{itemize}
\item $s_{x}w^{(k)}_{ij} = w^{(k)}_{ij}s_{x}$ for all $1\leqslant i, j\leqslant \dim(w^{(k)})$,
\item $t u^{xx}_{ij} = u^{xx}_{ij}t$ for all $1\leqslant i, j\leqslant \dim(u^{xx})$, where $t$ is some group-like element,
\item group-like relations.
\end{itemize}
\end{prop}

\begin{proof}
Let $\DD$ be the category of partitions generated by $\CC_{x}\ast NC_{ev}$ and $q_{k}$. By Lemma \ref{lem:xsimplepartitions}, all $x$-simple projective partitions in $\CC$ lie in some $\DD\subset \CC'\subset \CC$ obtained by adding commutation relations. Consider now a projective partition where the extreme points of each row are connected and coloured by $x$. Then as in the proof of Proposition \ref{prop:symmetricbybistochastic}, the presence of this partition is equivalent to the presence of all the partitions $t_{i} = b_{i}\otimes\pi(x, x)\otimes\pi(x,x)\otimes b_{i}^{*}$ which correspond to commutation relations as in the statement. The proof is then concluded as usual.
\end{proof}

\subsubsection{Third case}

The last case to consider is quotients of $S_{N}^{+\prime}\ast S_{N}^{+\prime}$.

\begin{prop}
Assume that $P_{xx}, P_{yy}, \pi(xx, xx), \pi(yy, yy)\in \CC$. Then, $\G_{N}(\CC)$ is a quotient of $C(S_{N}^{+\prime})\ast C(S_{N}^{+\prime})$ by one or several of the following relations :
\begin{itemize}
\item $s_{x}u^{yy}_{ij} = u^{yy}_{ij}s_{x}$ for all $1\leqslant i, j\leqslant \dim(u^{yy})$,
\item $s_{y}u^{xx}_{ij} = u^{xx}_{ij}s_{y}$ for all $1\leqslant i, j\leqslant \dim(u^{xx})$,
\item group-like relations.
\end{itemize}
\end{prop}

\begin{proof}
The proof is the same as in the previous cases. Note that if $P_{x}\in \CC$, then $\pi(xx, x)\in \CC$ so that commutation with $u^{xx}$ implies commutation with $x$, i.e.~with all of $C(\G_{N}(\CC_{x}))$ (and similarly for $y$).
\end{proof}

As for all the other cases, they are, up to exchanging the colours, equivalent to one of the previous ones.

\subsection{The amalgamated case}

If we now assume that there are partitions with two colours, we will get a free wreath product of a pair. To prove this, all we have to do is to show that $\CC_{\Gamma_{d}}\subset \CC$ for some $d$.

\begin{lem}\label{lem:amalgamatedodd}
Let $\CC$ be a category of noncrossing partitions containing a partition with an odd block of size at least three and a partition with a two-coloured block. Then, $\CC$ contains $P_{xx}$, $P_{yy}$ and $\pi(xy, xy)$.
\end{lem}

\begin{proof}
Let $p\in \CC$ be a partition containing a block $b$ with two colours and let $q$ be a minimal full subpartition containing $b$ so that $q^{*}q\in \CC$ and the endpoints of $q$ are those of $b$. If $b$ has size at least three, then using the same operations as in the proof of Lemma \ref{lem:symetrichasfourblock} we see that $\pi(xy, xy)\in \CC$. If $b$ has size two, i.e.~is equal to $D_{xy}$, then rotating $q^{*}q$ yields $q' = r\otimes \pi(x, y)\otimes\pi(x, y)\otimes \overline{r}^{*}$ for some $r\in \CC(w, \emptyset)$. But then, $\pi(yy, yy)q'$ contains a four-block with two colours and we can conclude with the first part of the proof.
\end{proof}

The classification will now follow from Proposition \ref{prop:classificationwreath}. We need however to identify the subgroups $\Lambda\subset \Gamma_{d}$ appearing as one-dimensional representations. For $n$ dividing $d$ and $1\leqslant m\leqslant n-1$, let us denote by $\Lambda_{n, m}$ the subgroup of $\Gamma_{d}$ generated by $(s_{x}s_{y})^{n}$ and $(s_{x}s_{y})^{m}s_{x}$.

\begin{thm}
Let $\CC$ be a category of noncrossing partitions with at least one block of even size greater than or equal to three and at least one block with two colours. Then, $\G_{N}(\CC) = H_{N}^{+}(\Gamma_{d}, \Lambda_{n, m})$.
\end{thm}

\begin{proof}
The isomorphism with $H_{N}^{+}(\Gamma_{d}, \Lambda)$ is a direct consequence of Lemma \ref{lem:amalgamatedodd} and Proposition \ref{prop:classificationwreath}. As for the form of $\Lambda$, simply recall that any subgroup of $\Gamma_{d}$ is of the form $\Lambda_{n}$ or $\Lambda_{n, m}$. Since in the first case there is no odd partition in the category, we must be in the second case.
\end{proof}

\section{Summary and further questions}\label{sec:summary}

In this final section, we will restate our main results in a different form. This will give more insight on possible generalizations to an arbitrary number of colours. To start, let us point out that our classification can be given as an explicit list of categories of partitions. In fact, commutation relations are explicitly expressed with partitions in the proofs. As for the group-like relations, any non-through-block projective partition is equivalent to an alternating one so that we only have to add such partitions. Moreover, in all cases the group of one-dimensional representations is a quotient of $\Z_{2}\ast\Z_{2}$ and the only relations that one can add to this group or its quotients are : $\gamma_{x} = 1$, $\gamma_{y} = 1$ or $(\gamma_{x}\gamma_{y})^{n} = 1$ for some integer $n$. Thus, we have a process to produce all possible categories of noncrossing partitions on two self-inverse colours :
\begin{itemize}
\item Choose a free product of categories of partitions of two orthogonal easy quantum groups,
\item Add partitions corresponding to commutation relations or twisted amalgamation,
\item Add $P_{x}$, $P_{y}$ or an alternating partition of even size (this partition can be given a canonical form depending on the original free product).
\end{itemize}

If we have rather expressed our results in terms of quotients by relations, it is because we think it naturally suggests generalizations. To explain this, let us give a name to the fact that each colour is its own inverse, since this was used crucially in our work.

\begin{de}
A noncrossing partition quantum group is said to be \emph{of orthogonal type} if all the colours are their own inverses.
\end{de}

Gathering the results of Sections \ref{sec:orthogonal}, \ref{sec:bistochastic}, \ref{sec:hyperoctahedral} and \ref{sec:symmetric} would produce a long list of all noncrossing partition quantum groups of orthogonal type on two colours. However, we have highlighted in Section \ref{sec:structure} three basic constructions on the class of noncrossing partition quantum groups : (twisted) amalgamated free products, group-like relations and commutation relations between one-dimensional representations and higher-dimensional ones. Since free quantum groups are in some sense the simplest examples of noncrossing partition quantum groups, we will say that a noncrossing partition quantum group is \emph{elementary} if it can be obtained from free quantum groups using only these three operations. The complete classification can now be stated as follows :

\begin{thm}
Let $\G$ be a noncrossing partition quantum group of orthogonal type with two colours. Then, $\G$ is either elementary of a free wreath product of pair.
\end{thm}

\begin{rem}
This statement is less precise than what we have obtained. In fact, we have proven that we only need very few commutation relations : in all cases but the symmetric one, it is enough to consider commutation relations with a quantum subgroup of one of the factors. This means that all other commutation relations can be obtained by these simple ones and group-like relations. In the symmetric case however, we also need commutation relations with $w^{(k)}$.
\end{rem}

In this form, it naturally suggests the following question : is any noncrossing partition quantum group of orthogonal type either elementary or a free wreath product of pair ?

If we now drop the "orthogonal type" assumption, there is another basic operation to consider. Its origin is a description of the free unitary quantum group $U_{N}^{+}$ as a complex version of the free orthogonal group $O_{N}^{+}$ given by T. Banica in \cite{banica1997groupe}. It has later been generalized by P. Tarrago and M. Weber in \cite{tarrago2015unitary} as follows : let $\G$ be a compact quantum group together with a fundamental representation $u$, let $d\geqslant 0$ be an integer and let $z$ be the canonical generator of $C^{*}(\Z_{d})$ (with the convention $\Z_{0} = \Z$). Then, the $d$-free complexification of $\G$ is the quantum subgroup of $C(\G)\ast C^{*}(\Z_{d})$ generated by $uz$. One can similarly define the $d$-tensor complexification. These two constructions were used to describe all noncrossing partition quantum groups on two mutually inverse colours in \cite{tarrago2015unitary} (note that they use the word "free" to denote any quantum group which is just "noncrossing" in our sense) :

\begin{thm}[Tarrago-Weber]
Any noncrossing partition quantum group on two mutually inverse colours can be obtained from free orthogonal easy quantum groups or free wreath products by using complexifications.
\end{thm}

Again, this leads to a natural question : can any noncrossing partition quantum group be obtained from noncrossing partition quantum groups of orthogonal type and free wreath products of pairs using only complexifications ? The first step here would be to give an explicit construction of the category of partitions of a complexification from the original one. Even in the easy quantum group case, this is unclear as far as we know.

\bibliographystyle{amsplain}
\bibliography{../../../../quantum}

\end{document}